\documentclass[reqno,11pt,centertags]{article}
\usepackage{amsmath,amsthm,amscd,amssymb,latexsym,upref}
\date{\today}

\input epsf
\usepackage{epsfig}
\usepackage[T2A,OT1]{fontenc}
\usepackage[ot2enc]{inputenc}
\usepackage[russian,english]{babel}

\usepackage{epic,eepicemu}

\newcommand{\bbD}{{\mathbb{D}}}

\newcommand{\bbR}{{\mathbb{R}}}

\newcommand{\bbZ}{{\mathbb{Z}}}
\newcommand{\bbC}{{\mathbb{C}}}

\newcommand{\cA}{{\mathcal{A}}}
\newcommand{\cB}{{\mathcal{B}}}

\newcommand{\cE}{{\mathcal{E}}}
\newcommand{\cF}{{\mathcal{F}}}

\newcommand{\cJ}{{\mathcal{J}}}

\newcommand{\cO}{{\mathcal{O}}}

\newcommand{\cS}{{\mathcal{S}}}

\newcommand{\fA}{{\mathfrak{A}}}

\newcommand{\fK}{{\mathfrak{K}}}
\newcommand{\fa}{{\mathfrak{a}}}

\newcommand{\fj}{{\mathfrak{j}}}

\newcommand{\fv}{{\mathfrak{v}}}
\newcommand{\fw}{{\mathfrak{w}}}
\newcommand{\fz}{{\mathfrak{z}}}

\newcommand{\bC}{{\mathbf{C}}}
\newcommand{\ba}{{\mathbf{a}}}
\newcommand{\bb}{{\mathbf{b}}}
\newcommand{\bc}{{\mathbf{c}}}

\newcommand{\e}{{\epsilon}}
\renewcommand{\k}{\varkappa}

\newcommand{\z}{\zeta}

\newcommand{\vp}{{\vec{p}}}
\newcommand{\vq}{{\vec{q}}}

\newcommand{\oc}{\overset{\circ}}
\newcommand{\is}{\mathcal{IS}_E}

\renewcommand{\Im}{\text{\rm Im}\,}
\newcommand{\GSMP}{\text{\rm GSMP}}
\newcommand{\KS}{\text{\rm KS}}
\newcommand{\KSA}{\text{\rm KSA}}
\newcommand{\tr}{\text{\rm tr}\,}

\newcommand{\dist}{\text{\rm dist}}
\newcommand{\Res}{\text{\rm Res}\,}

\allowdisplaybreaks \numberwithin{equation}{section}
\newtheorem{theorem}{Theorem}[section]

\newtheorem{lemma}[theorem]{Lemma}
\newtheorem{proposition}[theorem]{Proposition}

\newtheorem{corollary}[theorem]{Corollary}

\theoremstyle{definition}
\newtheorem{definition}[theorem]{Definition}
\newtheorem{notation}[theorem]{Notation}

\newtheorem{remark}[theorem]{Remark}

\date{\today}
\title
{Killip-Simon problem and Jacobi flow on GSMP matrices}
\author{B. Eichinger and P. Yuditskii\thanks{Supported by the Austrian Science Fund FWF, project no: P22025-N18.}}

\begin{document}

\maketitle
\begin{abstract}
One of the first and therefore most important theorems in perturbation theory claims that for an arbitrary self-adjoint operator $A$ there exists a perturbation $B$ of Hilbert-Schmidt class with arbitrary small operator norm, which destroys completely the absolutely continuos (a.c.) spectrum of the initial operator $A$ (von Neumann). However, if $A$ is the discrete free 1-D Schr\"odinger operator and $B$ is an arbitrary Jacobi matrix (of Hilbert-Schmidt class) the a.c. spectrum remains perfectly the same, that is, the interval $[-2,2]$. Moreover, Killip and Simon described explicitly  the spectral properties for such $A+B$. Jointly with Damanik they generalized this result to the case of perturbations of periodic Jacobi matrices in the non-degenerated case. Recall that the spectrum of a periodic Jacobi matrix is a system of intervals of a very specific nature. Christiansen, Simon and Zinchenko posed in a review dedicated to F. Gesztesy (2013)
 the following question: ``is there an extension of the Damanik-Killip-Simon theorem to the general finite system of intervals case?" In this paper we solve this problem completely. Our method deals with the Jacobi flow on GSMP matrices. GSMP matrices are probably  a new object in the spectral theory. They form a certain Generalization of matrices related to the Strong Moment Problem, the latter ones are a very close relative of Jacobi and CMV matrices. The Jacobi flow on them is also a probably new member of the rich family of integrable systems. Finally, related to Jacobi matrices of Killip-Simon class, analytic vector bundles and their curvature play a certain role in our construction and, at least on the level of ideology, this role is quite essential.
\end{abstract}

\section{Introduction}

\subsection{Main result}

\noindent
\textit{(1) Von Neumann Theorem} \cite{vN} states that for an arbitrary self-adjoint operator $A$, having a nontrivial absolutely continuous (a.c.) component of the spectrum, there exists a self-adjont perturbation $\delta A$ of Hilbert-Schmidt class  such that $A+\delta A$ has a pure point spectrum. Moreover, $\delta A$ may have
 an arbitrary small operator norm.

Therefore, the following result is already quite non-trivial.

\noindent
\textit{(2) Deift-Killip Theorem} \cite{DK}. For a discrete one-dimensional Schr\"odinger operator with square summable potential, the absolutely continuous part of the spectrum is $[-2,2]$.

Thus, under a special perturbations of Hilbert-Schmidt class (the square summable potential) the absolutely continuous spectrum of the free, discrete 1-D Schr\"odinger operator  is perfectly preserved. 
It is totally surprising that one can find a \textit{complete explicit characterization} of the spectral data if the perturbation is an arbitrary  Jacobi matrix of Hilbert-Schmidt class.

\noindent
\textit{(3) Killip-Simon Theorem} \cite{KS}. Let $d\sigma$ be a probability measure  on $\bbR$ with bounded but infinite support. As it is well known the orthonormal polynomials $P_n(x)$ with respect to this measure obey a three-term recurrence relation
\begin{equation}\label{eq1}
xP_n(x)=a(n)P_{n-1}(x)+b(n)P_n(x)+a(n+1) P_{n+1}(x), \quad a(n)>0.
\end{equation}
The following are equivalent:
\begin{itemize}
\item[(op)] $\sum_{n\ge 1}|a(n)-1|^2<\infty$ and $\sum_{n\ge 0}|b(n)|^2<\infty$.
\item[(sp)] The measure $d\sigma$ is supported on $[-2,2]\cup X$, and moreover
\begin{equation}\label{eq2}
\int_{-2}^2|\log\sigma'(x)|\sqrt{4-x^2}dx+\sum_{x_k\in X}\sqrt{x_k^2-4}^3<\infty.
\end{equation}
\end{itemize}
\begin{remark}
Of course the (op)-condition means that the Jacobi matrix
$$
J_+=\begin{bmatrix}
b(0)&a(1)& & \\
a(1)&b(1)& a(2)& \\
 &\ddots&\ddots&\ddots
\end{bmatrix}
$$
represents a Hilbert-Schmidt class perturbation of the matrix $\oc J_+$ with the constant coefficients 
$\oc a(n)=1$ and $\oc b(n)=0$. In this case we consider $J_+$ as an operator acting in the standard space of one-sided sequences $\l^2_+$. In its turn, the (sp)-condition means that the related \textit{spectral measure} $d\sigma$ has an absolutely continuous component supported on $[-2,2]$. Moreover, the spectral density $\sigma'(x)$ with respect to the Lebesgue measure satisfies an explicitly given integral condition, which in particular means that $\sigma'(x)\not=0$ a.e. on this interval. Besides that, the measure may have at most countably many mass points (the set $X$) outside of the given interval. Again, the corresponding set $X$ satisfies an explicitly given condition, which in particular means that the only possible accumulation points of this set are the endpoints $\pm 2$. Finally, note that there is no restriction on the \textit{singular component} of the measure $d\sigma$ on the interval $[-2,2]$.
\end{remark}

Later, also in a paper, which was published in Annals, the authors jointly with David Damanik generalized their result on the case of perturbations of \textit{periodic}  Jacobi matrices. To state this theorem we need a couple of definitions.

We define a distance between two one-sided sequences $b=\{b(n)\}_{n\ge 0}$ and
$\tilde b=\{\tilde b(n)\}_{n\ge 0}$ from $\l_+^\infty$ by
\begin{equation}\label{eq3}
\dist^2(b,\tilde b)=\dist_{\eta}^2(b,\tilde b):=\sum_{n\ge 0}|b(n)-\tilde b(n)|^2\eta^{2n}, \quad \eta\in(0,1).
\end{equation}
The distance $\dist (J_+,\tilde J_+)$ between two Jacobi matrices is defined via the distances between the generating coefficient sequences. 

Let $J(E)$ be the isospectral set of periodic two-sided Jacobi matrices with a given spectral set $E\subset\bbR$.
 The distance between $J_+$ and $J(E)$ is defined in a standard way
\begin{equation}\label{eq4}
\dist(J_+,J(E))=\inf\{\dist(J_+,\oc J_+):\ \oc J\in J(E)\},
\end{equation}
where $\oc J_+$ is the restriction of a two-sided matrix $\oc J$ on the positive half-axis.

\noindent
\textit{(4) Damanik-Killip-Simon Theorem} (DKST) \cite{KSDp}. Assume that $J_+$ is a Jacobi matrix and let $d\sigma$ be the associated spectral measure.
The following are equivalent:
\begin{itemize}
\item[(opp)]  Let $S_+$ denote the shift operator in $\l^2_+$. Then
\begin{equation}\label{opp}
\sum_{n\ge 0}\dist^2((S_+^*)^n J_+S_+^n, J(E))<\infty.
\end{equation}
\item[(spp)] The measure $d\sigma$ is supported on $E\cup X$, and moreover
\begin{equation}\label{eq5}
\int_{E}|\log\sigma'(x)|\sqrt{\dist(x,\bbR\setminus E)}dx+\sum_{x_k\in X}\sqrt{\dist(x_k, E)}^3<\infty.
\end{equation}
\end{itemize}

\begin{remark}
Note that \eqref{opp}  means that the shifts of the given Jacobi matrix $J_+$ approach to the isospectral set $J(E)$, but possibly not to a specific element $\oc J$ of this set. In the same time  \eqref{eq5} looks as a  straightforward counterpart of the condition \eqref{eq2}.
\end{remark}
\begin{remark}
Let us point out that the spectral set of any periodic two-sided Jacobi matrix $\oc J$ is a system of interval of a very special nature: the  system of intervals $E=[\bb_0,\ba_0]\setminus\cup_{j=1}^g(\ba_j,\bb_j)$ represents the spectrum of a periodic Jacobi matrix if and only if $E=T_m^{-1}([-2,2])$, where  $T_m(z)$ is a polynomial   with only real critical points, that is,
$$
T_m'(c)=0 \ \text{for} \ c\in\bbR,
$$
and its critical values $T_m(c)$ obey the conditions $|T_m(c)|\ge 2$. Actually, the Damanik-Killip-Simon Theorem was proved under an additional regularity condition $|T_m(c)|>2$ for all critical points $c$. In this case the degree $m=g+1$.
\end{remark}

The paper \cite{jreview} reviews recent progress in the understanding of the class of so called  \textit{finite gap} Jacobi matrices and their perturbations. In the end of the article the authors posed the following question: ``Is there an extension of the Damanik-Killip-Simon theorem to the general finite system of intervals $E$ case?"  In the present paper \textit{we solve completely this problem}, see Theorem \ref{mainhy} below.

Finite gap Jacobi matrices were discovered in 
the context of approximation theory
\cite{AKH, Akh60}, \cite[Chapter X]{AKHef}. They became especially famous because of their relation with the theory of integrable systems, for  historical comments we would refer to \cite{MaT} with many references therein. But the true meaning of this class was significantly clarified recently by C. Remling (in a paper, which  was also published in Annals): for a system of intervals $E$ the finite gap class $J(E)$ consists of all limit points of Jacobi matrices with an essential spectrum on $E$, having this $E$ as the support of their a.c. spectrum.

\noindent
\textit{(5) Remling Theorem} \cite{REMA11}. Let $E$ be a system of intervals.
Let $J_+$ be a Jacobi matrix with the generating coefficient sequences $\{a(n),b(n)\}$ such that its spectrum $\sigma(J_+)=E\cup X$, where $X$ is a set of isolated points, which accumulate only to the endpoints of the intervals, and $\sigma'(x)\not=0$ for a.e. $x\in E$. If 
$$
\oc a(n)=\lim_{m_k\to +\infty} a(n+m_k),\quad \oc b(n)=\lim_{m_k\to +\infty} b(n+m_k),
$$
for all $n\in\bbZ$, then the corresponding two-sided Jacobi matrix $\oc J$ belongs to $J(E)$.

Note that the system of shifts $\{(S_+^*)^nJ_+S_+^n\}_{n\ge 0}$ forms a precompact set in the compact-open topology (generated by the distance \eqref{eq3}).

For $E=[\bb_0,\ba_0]\setminus\cup_{j=1}^g(\ba_j,\bb_j)$ the class $J(E)$ represents a $g$-dimensional torus, which can be parametrized explicitly.

\noindent
\textit{(6) Baker-Akhiezer parametrization for the class $J(E)$}, see e.g. \cite[Theorem 9.4]{GT}.
For $\alpha\in \bbR^g/\bbZ^g$ let 
\begin{eqnarray}
\cA(\alpha)&=&\bar a^2\frac{\theta(\alpha+\mu+\bar\alpha)\theta(\alpha-\mu+\bar\alpha)}{\theta(\alpha+\bar\alpha)^2},
\label{aal}\\
\cB(\alpha)&=&\bar b+\partial_\xi\ln\frac{\theta(\alpha-\mu+ \bar\alpha)}{\theta(\alpha+\bar\alpha)}\label{bal},
\end{eqnarray}
 where 
 $$
 \theta(z)=\theta(z,\Omega)=\sum_{n\in \bbZ^g}e^{\pi i\langle \Omega n,n \rangle+2\pi i\langle z,n \rangle}, \quad z\in \bbC^g,
 $$
 with the following system of parameters depending on $E$: 
 \begin{itemize}
 \item $\Omega$ is a symmetric $g\times g$ matrix with a positive imaginary part, $\Im \Omega >0$; 
 \item $\bar\alpha\in\bbC^g$ is an appropriate shift; 
 \item $\mu\in\bbR^g/\bbZ^g$ and $\xi\in\bbR^g$ are certain fixed directions of discrete and continuous translations  on the torus $\bbR^g/\bbZ^g$, respectively; 
 \item $\bar a>0$ and $\bar b\in\bbR$ are normalization constants. 
 \end{itemize}
 Then $\oc J\in J(E)$ if and only if
 \begin{equation}\label{eq}
\oc a(n)^2=\cA(\alpha-\mu n), \quad
\oc b(n)=\cB(\alpha-\mu n),
\end{equation}
for some $\alpha\in \bbR^g/\bbZ^g$. In this case we write $\oc J=J(\alpha)$. Thus,
\begin{equation}\label{param1}
J(E)=\{J(\alpha):\ \alpha\in\bbR^g/\bbZ^g\}.
\end{equation}

\begin{definition}
For an arbitrary finite system of intervals $E$, we say that a Jacobi matrix $J_+$ belongs to the Killip-Simon class  $\KS(E)$ if  for some $X$ the corresponding spectral measure $d\sigma$ is supported on
$E\cup X$ and obeys \eqref{eq5}.
\end{definition}
\begin{theorem}\label{mainhy}
$J_+$ belongs to $\KS(E)$ if and only if there exist $\epsilon_\alpha(n)\in\l_+^2(\bbR^g)$ and $\epsilon_a(n)\in \l^2_+$,
$\epsilon_b(n)\in \l^2_+$ such that
 \begin{eqnarray}
a(n)^2&=&\cA(\sum_{k=0}^n\epsilon_\alpha(k)-\mu n)+\epsilon_a(n),\label{132}\\
b(n)&=&\cB(\sum_{k=0}^n\epsilon_\alpha(k)-\mu n)+\epsilon_b(n),\label{133}
\end{eqnarray}
where $\cA(\alpha)$ and $\cB(\alpha)$ are defined in \eqref{aal} and \eqref{bal}, respectively.
\end{theorem}
\begin{remark} In the one interval case the functions $\cA$ and $\cB$ are constants, e.g. if $E=[-2,2]$, then $\cA=1$ and $\cB=0$ and we obtain the original Killip-Simon Theorem.
\end{remark}

\begin{remark}\label{rem7}
It is easy to see that a Jacobi matrix of the form \eqref{132}-\eqref{133} satisfies \eqref{opp}, see Lemma \ref{lem72}. Moreover, from our explicit formulas one can give immediately a suitable approximant  for $(S_+^*)^nJ_+S^n_+$, this is
$J(\alpha_n)\in J(E)$, $\alpha_n=\sum_{k=0}^n\epsilon_\alpha(k)-\mu n$; or conclude that, if the series $\beta:=\sum_{k=0}^\infty\epsilon_\alpha(k)$ conditionally converges, then the coefficients  of $J_+$ approach, in fact, to the coefficients of the fixed element  $J(\beta)\in J(E)$,
$$
a^2(n)-\cA(\beta-\mu n)\to 0\ \text{and}\ b(n)-\cB(\beta-\mu n)\to 0,\ \text{where}\ n\to\infty.
$$
This representation is a little bit ambiguous, for this reason see Remark \ref{rem71}.
\end{remark}

\subsection{Basic ideas of the method and the structure of the paper}

\noindent
\textit{The proof of DKST}   was based on two things:
\begin{itemize}
\item[(i)] Magic formula for periodic Jacobi matrices
\item[(ii)] Matrix version of the Killip-Simon theorem
\end{itemize}

The first one is the following identity. Let $S$ be the shift in the space of two sided sequences $\l^2$. If 
$E=[\bb_0,\ba_0]\setminus\cup_{j=1}^g(\ba_j,\bb_j)=T^{-1}_{g+1}([-2,2])$, then
\begin{equation}\label{mf}
T_{g+1}(\oc J)=S^{g+1}+S^{-(g+1)}
\end{equation}
for all $\oc J\in J(E)$. The last matrix can be understood as the $(g+1)\times(g+1)$-block Jacobi matrix with the constant block coefficients $\oc A(n)=I_{g+1}$ and $\oc B(n)=\mathbf{0}_{g+1}$.

Now, for $J_+$ the matrix $T_{g+1}(J_+)$ is a $(2g+3)$-diagonal matrix, or, also a one-sided $(g+1)\times(g+1)$ Jacobi block-matrix
$$
T_{g+1}(J_+)=\begin{bmatrix}
B(0)&A(1)& & \\
A(1)&B(1)& A(2)& \\
 &\ddots&\ddots&\ddots
\end{bmatrix}.
$$
Such matrix has a spectral $(g+1)\times(g+1)$ matrix-measure, say $d\Sigma$. According to \cite{KSDp} the matrix analog of \eqref{eq2} is of the form
\begin{equation}\label{eq12}
\int_{-2}^2|\log\det\Sigma'(y)|\sqrt{4-y^2}dy+\sum_{y_k\in Y}\sqrt{y_k^2-4}^3<\infty,
\end{equation}
as before  $[-2,2]\cup Y$ is the support of $d\Sigma$. On the one hand this condition can be rewritten by means of the spectral measure $d\sigma$ of the initial Jacobi matrix $J_+$ into the form \eqref{eq5}, $y=T_{g+1}(x)$. On the other hand, due to the matrix version of the Killip-Simon theorem, \eqref{eq12} is equivalent to $T_{g+1}(J_+)-(S_+^{g+1}+(S_+^*)^{g+1})$ belongs to the Hilbert-Schmidt class. This is a certain bunch of conditions on the coefficients of $J_+$, but we should recognize that extracting from this simple-looking condition the final one \eqref{opp}, is a very non-trivial task.

\smallskip
Our \textit{first basic observation} is the following.

\begin{lemma}
For a system of intervals $E$  there exists a unique rational function $V(z)$ such that
\begin{equation}\label{eq13}
E=[\bb_0,\ba_0]\setminus\bigcup_{j=1}^g(\ba_j,\bb_j)=V^{-1}([-2,2]),
\end{equation}
and $\Im V(z)>0$ for $\Im z>0$.
\end{lemma}
\begin{proof}
Let $\Delta(z)$ be the Alphors function in the domain $\bar\bbC\setminus E$. Among all analytic functions in this domain, which vanish at infinity and are bounded by one in absolute value, this function has the biggest possible value Cap$(E)=|z\Delta(z)|_{z=\infty}$ (the so-called analytic capacity) \cite{ALF}. As it is well known
\begin{equation}\label{eq14}
\frac{1-\Delta(z)}{1+\Delta(z)}=\sqrt{\prod_{j=0}^g\frac{z-\ba_j}{z-\bb_j}}.
\end{equation}
Then
\begin{equation}\label{eq15}
V(z)=\frac{1}{\Delta(z)}+\Delta(z)=\lambda_0z+\bc_0+\sum_{j=1}^g\frac{\lambda_j}{\bc_j-z}
\end{equation}
where $\lambda_j>0$, $j\ge 0$, and $\Delta(\bc_j)=0$, $\bc_j\in(\ba_j,\bb_j)$, $j\ge 1$.
\end{proof}

Note that in this proof we represented $V(z)$ as a superposition of a function $\Delta:\bar\bbC\setminus E\to \bbD$ with the Zhukovskii map. Essentially \eqref{eq15} is our \textit{generalized magic formula}, though it holds of course not for Jacobi matrices.

\bigskip
\noindent
\textit{Jacobi, CMV and SMP matrices}. Jacobi matrices probably the oldest object in the spectral theory of self-adjoint operators generated by the moment problem \cite{AKHmp}
\begin{equation}\label{eq16}
s_k=\int x^k d\sigma.
\end{equation}
In this problem we are looking for a measure $d\sigma$ supported on the real axis, which provides the representation 
\eqref{eq16} for the given moments $\{s_k\}_{k\ge 0}$. In this sense CMV matrices are related to the \textit{trigonometric moment problem}, which corresponds to the same question with respect to a measure supported on the unit circle. Note that this problem is also classical \cite{AKHmp}, but corresponding CMV matrices are a comparably fresh object in the spectral theory \cite{2005v1, 2005v2}. The \textit{strong moment problem} corresponds to measures on the real axis in the case that the moments are given for \textit{all integers} $k$.  An extensive bibliography of works on the strong moment problem can be found in the survey  \cite{JN}, concerning its matrix generalization see \cite{Sim1,Sim2}. 

As usual the solution of the problem deals with the orthogonalization of the generating system of functions, that is, the system
$$
1, \frac {-1} x, x, \frac {(-1)^2}{x^2}, x^2,\dots
$$
in the given case. The multiplication operator by the independent variable in $L^2_{d\sigma}$ with respect to the related \textit{orthonormal basis we call SMP matrix} (this is exactly the way of the appearance of Jacobi and CMV matrices in connection with the power and trigonometric moment problem, respectively). In another terminology they are called Laurent-Jacobi matrices \cite{BD2,DUD,HN}. Very similar to the CMV-case, this is a five-diagonal matrix of a special structure, say $A_+=A_+(d\sigma)$. We assume that the measure is compactly supported and the origin does not belong to the support of this measure. In this case our $A_+$ is bounded, moreover $A_+^{-1}$ is also a bounded operator of a similar five-diagonal structure (just shifted by one element!)

Note that, by a linear change of variable, we can always normalize an arbitrary \textit{two intervals} system to the form
$\bc_1=0$, that is,
\begin{equation}\label{eq17}
E=[\bb_0,\ba_0]\setminus(\ba_1,\bb_1)=V^{-1}([-2,2]),\quad V(z)=\lambda_0+\bc_0-\frac{\lambda_1}{z}.
\end{equation}
Without going in details, dealing with the structure of SMP matrices, we can formulate our \textit{second basic observation}. 
\begin{proposition}\cite{EPY}
Let $A(E)$ be the set of all two sided SMP matrices of period two with their spectrum on $E$ \eqref{eq17}.
Then $\oc A\in A(E)$ if and only if 
\begin{equation}\label{eq18}
V(\oc A)=\lambda_0\oc A+\bc_0-\lambda_1 (\oc A)^{-1}=S^2+S^{-2}.
\end{equation}
\end{proposition}

\begin{remark} 
It is highly important in \eqref{eq18} to be hold  that both $\oc A$ and $(\oc A)^{-1}$ are five-diagonal matrices.
\end{remark}

Naturally, \eqref{eq17}-\eqref{eq18} have to be generalized to the multi-interval case. This leads to the concept of GSMP
matrices (G for generalized), see the next subsection. However, even after such a generalization the result on spectral properties of ("some") GSMP matrices of Killip-Simon class would be interesting probably only
to a small circle of specialists, working with the strong moment problem. The point is that GSMP matrices are used here as a certain intermediate (but very important) object. In a sense, this is the best possible choice of a \textit{system of coordinates}. We can try to clarify the last sentence. The standard point of view on $J(E)$ is to associate $J(E)$ with the hyperelliptic Riemann surface $\mathfrak{R}_E=\{(z,w):\ w^2=\prod_{j=0}^g(z-\ba_j)(z-\bb_j)\}$. Then $J(E)$ corresponds to the "real part" of the Jacobian variety Jac$(\mathfrak{R}_E)$ of this surface see e.g. \cite{MaT, MUM}.  Periodic GSMP matrices, satisfying 
\begin{equation}\label{magic}
V(\oc A)=S^{g+1}+S^{-(g+1)}
\end{equation}
for $V(z)$ given in \eqref{eq15}, are most likely the best possible choice for a coordinate system on the affine part of Jac$(\mathfrak{R}_E)$, at least in application to spectral theory.

Thus, the point is to go back to Jacobi matrices. Let $d\sigma$ be compactly supported and $0$ does not belong to its support. We can define the map 
$$
\cF_+: \text{SMP}\to \text{Jacobi}
$$ just setting $J_+(\sigma)$ in correspondence with the given $A_+(\sigma)$.
If so, we can define (in a naive way) a discrete dynamical system (\textit{Jacobi flow on SMP matrices}) by the map $\cJ_+$, which corresponds to the following commutative diagram:
\begin{equation}\label{defjfg+}
\begin{array}{ccc}
 \text{SMP} & \xrightarrow{\mathcal J_+}  & \text{SMP}  \\
   &   &    \\
_{\mathcal F_+}  \big\downarrow &   &  _{\mathcal F_+} \big\downarrow  \\
   &   &    \\
\text{Jacobi}  &  \xrightarrow{\mathcal S_+}  &  \text{Jacobi} 
\end{array}
\end{equation}
where $\cS_+ J_+=S^*_+ J_+ S_+$.

The  \textit{third basic observation} deals with the idea of getting properties of the class $\KS(E)$ from the corresponding properties of the class of SMP (or, generally, GSMP) matrices using the above introduced dynamical system
$$
A_+(n)=\cJ_+^{\circ n}(A_+).
$$
 
 This definition \eqref{defjfg+} is naive for the following reason. In the transformation $J_+(n)=\cS_+^{\circ n} J_+$ the eigenvalues in the gaps start to move. E.g., in the generic case for an initial $\oc J_+$, which corresponds to one of our fundamental  operators $\oc J\in J(E)$, the eigenvalues will cover densely  the spectral gaps $(\ba_j,\bb_j)$. Thus, corresponding to such measures $A_+(n)$ just can not be properly defined. The easiest way to explain that nevertheless our program is doable is the following:  pass to two-sided Jacobi matrices and enjoy unitarity of the shift $S$ in $\l^2$! (One can  actually work with one-sided matrices but use methods related to two dimensional cyclic subspaces, which is naturally required if one works with two-sided matrices).

In the next subsection we give  formal definitions for GSMP matrices and the Jacobi flow on them,  but probably we can already outline the structure of the current paper:

Section 2. We recall the functional model for finite gap Jacobi matrices. In this model each operator is marked by a Hardy
space $H^2(\alpha)$ of character-automorphic functions  in the domain $\bar\bbC\setminus E$, where $\alpha$ is a character of the fundamental group of this domain \eqref{aleqal}, so, as before, $\alpha\in\bbR^g/\bbZ^g$ cf. \eqref{param1}. Here $J(\alpha)$ is the multiplication operator by the independent variable with respect to the basis $\{e^\alpha_n\}_{n\in\bbZ}$
\eqref{defeal}, and $\{e^\alpha_n\}_{n\ge 0}$ is an intrinsic basis in $H^2(\alpha)$.  The point is that the in this domain \textit{inner} function $\Delta(z)$ and the fixed ordering $\bC=\{\bc_1,...,\bc_g\}$ of its zeros generates another natural basis $\{f^\alpha_n\}_{n\ge 0}$ in $H^2(\alpha)$ \eqref{gsmpbase}. Thus, we obtain a new family of operators
$$
A(E,\bC)=\{A(\alpha,\bC):\ \alpha\in\bbR^g/\bbZ^g\}.
$$
This is the collection of all periodic GSMP matrices associated with the given spectral set $E$ and a fixed ordering $\bC$ of zeros of the Alphors function $\Delta(z)$. The fact that $\Delta(z)$ is \textit{single valued} (the character corresponding to this function is trivial) is responsible for the periodicity of an arbitrary $A(\alpha,\bC)$. 

Another characteristic feature of $\Delta(z)$ is its certain conformal invariance. Indeed, if $w=w_j=\frac{1}{c_j-z}$, then $\Delta_j(w):=\Delta(z)$ is the Alphors function in the $w$-plane. The given ordering $\bC$ generates the specific ordering
$$
\bC_j=\left\{\frac 1{\bc_{j+1}-\bc_j},\dots,\frac 1{\bc_{g}-\bc_j},0,\frac 1{\bc_{1}-\bc_j},\dots,\frac 1{\bc_{j-1}-\bc_j}\right\}
$$
and the multiplication by $w$ with an appropriate shift is again a \textit{periodic GSMP matrix}. In other words,
\begin{equation}\label{aut1}
S^{-j}(\bc_j-A(\alpha,\bC))^{-1}S^j\in A(E_j,\bC_j),
\end{equation}
where  $E_j=\{y=\frac 1{c_j-x}:\ x\in E\}$. Note that $0=w(\infty)$.
Let us point out that the spectral condition \eqref{eq5} possesses the same conformal invariance property.
Thus, passing from the $e$-basis to the $f$-basis in $H^2(\alpha)$, we payed a certain prize: $J(\alpha)$ is three diagonal and  $A(\alpha,\bC)$ is
a $(2g+3)$-diagonal matrix. 
 In the same time we essentially win, since $(\bc_j-J(\alpha))^{-1}$ has infinitely many non-trivial diagonals, but due to \eqref{aut1} all matrices
$(\bc_j-A(\alpha,\bC))^{-1}$ are still $(2g+3)$ diagonal.
For them \eqref{eq15} (in the chosen basis) is nothing but the magic formula \eqref{magic}.

The Jacobi flow on $A(E,\bC)$ can be defined in a very natural way. Since
$S^{-1}J(\alpha) S=J(\alpha-\mu)$, we set 
$$
\cJ A(\alpha,\bC)=A(\alpha-\mu,\bC).
$$
As we see, this is just one new object in the family of integrable systems.

Therefore, thanks to this section we are well prepared to understand the structure of GSMP matrices, $A\in \GSMP(\bC)$, and the Jacobi flow on them, $A(n)=\cJ^{\circ n}A$, in the general case, which is done in the Sections 3 and 4, respectively.

Section 5. Thanks to the block-matrix version of the Killip-Simon theorem, it is a fairly simple task to write the necessary and sufficient condition for $A\in\GSMP(\bC)$ with the spectral data \eqref{eq5}  in the form 
\begin{equation}\label{inhs}
V(A)-(S^{-(g+1)}+S^{g+1})\  \text{is in the Hilbert-Schmidt class}.
\end{equation}
Or, equivalently,
\begin{equation}\label{inhs2}
H_+(A)<\infty
\end{equation}
for the \textit{Killip-Simon functional of the problem}, which is basically the $\l^2_+$-part of the trace of $(V(A)-(S^{-(g+1)}+S^{g+1}))^2$, for the precise expression see \eqref{eq10}. In the spirit of our third basic observation, we compute the "derivative" of this functional in the direction of the Jacobi flow, that is, the value
$$
\delta_{\cJ}H_+(A):=H_+(A)-H_+(\cJ A),
$$ 
see Lemma \ref{lemjder}. \textit{This derivative represents a finite sum of squares!} Now, we can rewrite \eqref{inhs2} as the "integral" $\sum_{n\ge 0}\delta_{\cJ}H_+(\cJ^{\circ n}A)<\infty$ to get  certain $\l^2$-properties, which are already more related to the Jacobi matrix $J=\cF A$ than to the given GSMP matrix $A$ itself.

Section 6. But all this was related to the coefficients of $V(A)$, not to the ones of $A$ (or the system of iterates $A(n)$, to be more precise). This is probably the hardest technical part of the work. To indicate the difficulty, we would mention the following. In \cite{NPVY} we found higher-order generalizations of Killip-Simon sum rules (relations between coefficients of $J_+$ and the spectral measure $d\sigma$), for a \textit{single interval spectrum}. But only for a very special family (related to Chebyshev polynomials of an arbitrary degree $n$), which was initially found in \cite{LNS}, we were able to convert the result of the form \eqref{inhs} to explicit relations on the coefficients of the given $J_+$. Otherwise, each particular case becomes a reason for an interesting research, see e.g. \cite{K2004, GZ,SZ}. Moreover, a nice looking general conjecture was recently disproved by M. Lukic \cite{LU}. By the  way, for a highly interesting new development in this area see \cite{GNR}. In this section we prove Theorem \ref{th73}. Practically, this is already a parametric representation for coefficients of Jacobi matrices of $\KS(E)$. 

Section 7. In this section we finalize the parametric representation for Killip-Simon Jacobi matrices associated to an arbitrary system of intervals $E$, that is, we prove the main Theorem \ref{mainhy}. In the end of this section we demonstrate implicitly our \textit{last basic for this paper observation} that the spectral theory in the spirit of \cite{CD} could be more powerful than the  classical orthogonal polynomials approach \cite{AKHmp, BS}, see especially Subsection \ref{subs72}. Explicitly this was demonstrated in \cite{PY, VY, PVY}.
At the moment we are not able to present a theory of spaces of vector bundles, which corresponds as  model spaces to Jacobi matrices of Killip-Simon class even in a finite gap case. 

\subsection{GSMP matrices and Jacobi flow on them in solving the Killip-Simon problem}

In this subsection we give formal definitions for the named objects so that in the end of it we are able to state Theorem \ref{th73}. This is the main ingredient in our proof of Theorem \ref{mainhy}.

Let $\{e_n\}$ be the standard basis in $\l^2$. Depending on the context, $\l^2_+$ is the  set of square-summable one-sided sequences or the subspace of $\l^2$ spanned by $\{e_n\}_{n\ge 0}$. In the last case $\l^2_-:=\l^2\ominus \l^2_+$ and $P_+:\l^2\to \l^2_+$ is the orthogonal projector. 
Also $\{\delta_k\}_{k=0}^g$ denotes the standard basis in the Euclidian space $\bbC^{g+1}$.

By $T^*$ we denote the conjugated operator to an operator  $T$, or the conjugated matrix  if $T$ is a matrix.
In particular, for a vector-column $\vp\in \bbC^{g+1}$, $(\vp)^*$ is a $(g+1)$-dimensional vector-row. Consequently, the scalar product in $\bbC^{g+1}$ can be given in the following form
$$
\langle \vp,\vq \rangle=(\vq)^*\vp.
$$
The notation $T^-$ denotes the upper triangular part of a matrix $T$ (including the main diagonal), respectively 
$T^+:=T-T^-$ is its lower triangular part (excluding the main diagonal).

GSMP matrices form a certain special subclass of real symmetric $(2g+3)$-diagonal matrices, $g\ge 1$.
First of all, the class depends on an ordered collection of distinct points $\bC=\{\bc_1,\dots,\bc_g\}$. That is, if needed we will specify the notation $\GSMP(\bC)$. We will define two-sided GSMP matrices, but their restrictions on the positive half-axis will be highly important.

\begin{definition}
We say that $A$ is \textit{GSMP-structured} if it is a $(g+1)$-block Jacobi matrix
\begin{equation}\label{n1}
A=\begin{bmatrix}
\ddots&\ddots&\ddots&& &\\
&A^*(\vp_{-1})&B(\vp_{-1},\vq_{-1})&A(\vp_0)& & \\
& &A^*(\vp_{0})&B(\vp_{0},\vq_0)&A(\vp_1)& \\
& & &\ddots&\ddots&\ddots
\end{bmatrix}
\end{equation}
such that
\begin{equation}\label{n2}
A(\vp)=\delta_g \vp\,^*,
\quad
B(\vp,\vq)
=(\vq \vp\,^*)^-+(\vp\vq\,^*)^++\tilde\bC,
\end{equation}
and
\begin{equation}\label{n3}
\tilde \bC=\begin{bmatrix}
\bc_1& & & \\
& \ddots& & \\
& & \bc_g & \\
& & &0
\end{bmatrix},\ 
\vp_j=
\begin{bmatrix}
p^{(j)}_0\\
\vdots\\
p^{(j)}_g
\end{bmatrix}, \ 
\vq_j=
\begin{bmatrix}
q^{(j)}_0\\
\vdots\\
q^{(j)}_g
\end{bmatrix}, \quad p^{(j)}_g>0.
\end{equation}
 We call $\{\vp_j,\vq_j\}_{j\in\bbZ}$ the generating coefficient sequences (for the given $A$).
\end{definition}

\begin{remark}
Concerning the last condition in \eqref{n3}: actually, it is important that $p^{(j)}_g\not= 0$. The choice $p^{(j)}_g> 0$
is a matter of normalization. Further, throughout this paper we will assume in this definition that the much  stronger condition 
\begin{equation}\label{n5}
\inf_{j\in\bbZ}p^{(j)}_g>0
\end{equation}
holds. Note that these coefficients $\{p^{(j)}_g\}_{j\in\bbZ}$ form the non-trivial part of the last upper non-vanishing $(g+1)$-th diagonal of a GSMP-structured matrix $A$.
\end{remark}

\begin{definition} Let $S$ be the shift operator $Se_n=e_{n+1}$.
A GSMP-structured matrix $A$ belongs to the GSMP class if the matrices $\{\bc_k-A\}_{k=1}^g$ are invertible, and moreover
$S^{-k}(\bc_k-A)^{-1}S^k$ are GSMP-structured. To abbreviate we write $A\in \GSMP(\bC)$.
\end{definition} 

\begin{remark}
As it follows from the definition the entries of the last upper non-trivial $(g+1)$-th diagonal of the matrix $S^{-k}(\bc_k-A)^{-1}S^k$ should satisfy a counterpart of the condition \eqref{n5}. This set of conditions can be written explicitly by means of the coefficients of the initial GSMP-structured matrix $A$, see \eqref{altdef}.  Moreover, this set of conditions on the forming sequences $\{\vp_j,\vq_j\}_{j\in\bbZ}$ can be considered as a \textit{constructive definition}  of GSMP matrices, see 
Theorem \ref{defaltdef}. That is, $A\in\GSMP(\bC)$ if it is GSMP-structured and \eqref{altdef} holds for the generating sequences.
\end{remark}

Let $J$ be a Jacobi matrix with coefficients
$\{a(n),b(n)\}$:
\begin{equation}\label{ijf1}
Je_n=a(n) e_{n-1}+b(n) e_n +a(n+1) e_{n+1}, \quad a(n)>0, \ n\in\bbZ.
\end{equation}
The two-dimensional space spanned by $e_{-1}$ and $e_0$ forms a cyclic subspace for $J$. Also, $J$ can be represented as a two-dimensional perturbation of the orthogonal sum with respect to the decomposition $\l^2=\l^2_-\oplus \l^2_+$
\begin{equation}\label{ijf2}
J=\begin{bmatrix}
J_-& 0\\
0& J_+
\end{bmatrix}+a(0)(e_0\langle \cdot, e_{-1} \rangle+e_{-1}\langle \cdot, e_{0} \rangle ).
\end{equation}
We have a similar decomposition for $A\in\GSMP(\bC)$
\begin{equation}\label{ijf3}
A=\begin{bmatrix}
A_-& 0\\
0& A_+
\end{bmatrix}+\|\vp_0\|(\tilde e_0\langle \cdot, \tilde e_{-1} \rangle+\tilde e_{-1}\langle \cdot, \tilde e_{0} \rangle ),
\end{equation}
where
$$
\tilde e_{-1}=e_{-1}, \quad \tilde e_0:=\frac{1}{\|\vp_0\|}P_+ A e_{-1}. 
$$

\begin{definition}
For $A\in \GSMP$ the Jacobi matrix $J=\cF A$ is uniquely defined by the conditions
\begin{equation}\label{ijf4}
r_{\pm}(z):=\langle (J_{\pm}-z)^{-1} e_{\frac{-1\pm 1}{2}}, e_{\frac{-1\pm 1}{2}}\rangle=
\langle (A_{\pm}-z)^{-1}\tilde e_{\frac{-1\pm 1}{2}}, \tilde e_{\frac{-1\pm 1}{2}}\rangle
\end{equation}
and $a(0)=\|\vp_0\|$.
\end{definition}

\begin{definition} Let
$\cS J:= S^{-1}J S$. 
The Jacobi flow on GSMP matrices is generated by the transformation $\cJ$, which makes the following diagram commutative
\begin{equation}\label{defjfg}
\begin{array}{ccc}
 \text{GSMP} & \xrightarrow{\mathcal J}  & \text{GSMP}  \\
   &   &    \\
_\mathcal F  \big\downarrow &   &  _\mathcal F \big\downarrow  \\
   &   &    \\
\text{Jacobi}  &  \xrightarrow{\mathcal S}  &  \text{Jacobi} 
\end{array}
\end{equation}
The corresponding discrete dynamical system (Jacobi flow) is of the form $A(n+1)=\cJ A(n)$.
\end{definition}

The coefficients of the Jacobi matrix $J=\cF A$ are easily represented by means of the Jacobi flow acting on the initial $A$. Namely,

\begin{corollary}\label{cor113}
Let $J=\cF A$ and $A(n)=\cJ^{\circ n} A$. In the above notations
\begin{equation}\label{coefflow}
a(n)=\|\vp_0(n)\|, \quad b(n-1)=q_g^{(-1)}(n)p_g^{(-1)}(n).
\end{equation}
\end{corollary}

Now we can define the Killip-Simon class of GSMP matrices. Let $E$ be a system of $g+1$ disjoint intervals, $E=[\bb_0,\ba_0]\setminus \cup_{j=1}^g(\ba_j,\bb_j)$.
Let $V(z)=V_E(z)$ be the unique function, which was given in \eqref{eq15}.

\begin{proposition}\label{prop18}
 $\oc A\in\GSMP(\bC)$  belongs to the isospectral set  of periodic matrices $A(E,\bC)$ if and only if 
it obeys the magic formula \eqref{magic}.
\end{proposition}

\begin{definition} Let $A\in\GSMP(\bC)$. Let $\sigma_\pm$ be the related spectral measures, that is,
\begin{equation}\label{specmes}
r_{\pm}(z)=\int \frac{d\sigma_{\pm}(x)}{x-z},
\end{equation}
where $r_{\pm}(z)$ are given in \eqref{ijf4}.
We say that $A$ belongs to the Killip-Simon class
$\KSA(E, \bC)$ if
the measures $\sigma_\pm$ are supported on $E\cup X_{\pm}$, and both satisfy \eqref{eq5}.
\end{definition}

The following theorem is just a consequence of the matrix version of Killip-Simon theorem.

\begin{theorem}
$A\in\GSMP(\bC)$ belongs to the Killip-Simon class $\KSA(E,\bC)$ if
the difference $V_E(A)-(S^{-(g+1)}+S^{g+1})$ belongs to the Hilbert-Schmidt class.
\end{theorem}

However, the next statement is already highly non-trivial. Practically, it gives a parametrization of the coefficients of \textit{Jacobi matrices} of Killip-Simon class with the essential spectrum on $E$, see Remark \ref{rem22}.
\begin{theorem}\label{th73} 
For $A\in \GSMP(\bC)$, let $A(n+1)=\cJ A(n)$, $A(0)=A$ and let $\{\vp_j(n),\vq_j(n)\}_{j\in \bbZ}$ be the forming $A(n)$ coefficient sequences. The given $A$ belongs to $\KSA(E,\bC)$ if and only if
\begin{eqnarray}
\{p^{(-1)}_j(n)-p^{(0)}_j(n)\}_{n\ge 0}\in \l^2_+,&
 \{q^{(-1)}_j(n)-q^{(0)}_j(n)\}_{n\ge 0}\in \l^2_+
\label{m29}\\
\{\lambda_0p_g^{(0)}(n)-1\}_{n\ge0}\in \l_+^2, 
&\{\lambda_0\langle \vp_0(n),\vq_0(n) \rangle+\bc_0\}_{n\ge 0}\in \l_+^2
\label{m30}\\
&\{\Lambda_k (\vp_0(n),\vq_0(n))-\lambda_k\}_{n\ge 0}\in\l^2_+
\label{m31} 
\end{eqnarray}
hold for all $j=0,\dots, g-1$ and all $k=1,\dots, g$.
\end{theorem}

\begin{remark}\label{rem22}
We define the Killip-Simon class of Jacobi matrices by the spectral property \eqref{eq5}.
Since the spectral data of $J_+$ coincides with the spectral data of $A_+$,  a combination of 
Corollary \ref{cor113} (see \eqref{coefflow}), and the above theorem (see \eqref{m29}-\eqref{m31}) gives
 a parametric representation for matrices from $\KS(E)$. With a certain effort we can derive
 \eqref{132}-\eqref{133} from this. This is probably the central point of our GSMP matrices approach: from the very beginning we can see clearly a certain set of polynomials
 $\Lambda_k(\vp(n),\vq(n))$ which are constants modulo $\l^2_+$-sequences \eqref{m31}. So, as soon as such an expression appears in a numerator or a denominator of any rational function of the coefficients of $A(n)$, it can be replaced by a positive constant $\lambda_k$ (modulo an $\l^2_+$-sequence). Otherwise, it is really hard to imagine, how one can simplify conditions of the form \eqref{inhs} to that ones that deal with any explicit properties of coefficients of $J\in \KS(E)$, see the discussion on higher-order sum rules in the end of the previous subsection.
 \end{remark}

\section{Functional models for $J(E)$ and $A(E,\bC)$. \\ Jacobi flow on periodic GSMP matrices}

\subsection{ Hardy spaces and class $J(E)$}

In what follows, we will use functional models for the class of reflectionless matrices $J(E)$ in the form as considered in \cite{SY}. To this end, we need
to recall certain special functions related to function theory in
 the common resolvent domain $\Omega=\overline{\mathbb C}\setminus E$ for $J\in J(E)$. Note that in this case, $E$ can be a set of an essentially more complicated structure \cite{Has, Pom, WID}, than
 a system of intervals.  

Let $\mathbb D/\Gamma\simeq\overline{\mathbb C}\setminus E$ be a uniformization of the domain $\Omega$. It means that there exists a Fuchsian
group $\Gamma$ and a meromorphic function $\fz:\mathbb D\rightarrow\overline{\mathbb C}\setminus E$, $\fz\circ\gamma=\fz$ for all
$\gamma\in\Gamma$, such that
\begin{equation*}
 \forall z\in\overline{\mathbb C}\setminus E~\exists\zeta\in\mathbb D\!:~\fz(\zeta)=z \text{ and } \fz(\zeta_1)=\fz(\zeta_2)\Rightarrow
\zeta_1=\gamma(\zeta_2).
\end{equation*}
We assume that $\fz$ meets the normalization $\fz(0)=\infty$, $(\zeta \fz)(0)>0$.

Let $\Gamma^{*}$ be the group of characters of the discrete group $\Gamma$,
$$
\Gamma^*=\{\alpha|\ \alpha:\Gamma\to \bbR/\bbZ\ \text{such that}\ \alpha(\gamma_1\gamma_2)=\alpha(\gamma_1)+\alpha(\gamma_2)\}
$$
Since $\Gamma$ is formed by $g$ independent generators, say $\{\oc \gamma_j\}_{j=1}^g$, the group $\Gamma^*$ is equivalent to $\bbR^g/\bbZ^g$,
\begin{equation}\label{aleqal}
\alpha\simeq\{\alpha(\oc \gamma_1),\dots, \alpha(\oc\gamma_g)\}\in\bbR^g/\bbZ^g.
\end{equation}

\begin{definition}
 \label{def:htwo}
For $\alpha\in\Gamma^*$
we define the Hardy  space of character automorphic functions as
\begin{equation*}
H^2(\alpha) = H^2_{\Omega}(\alpha) = \{ f \in H^2\!:~ f \circ \gamma = e^{2\pi i\alpha(\gamma)} f,~\gamma\in\Gamma \},
\end{equation*}
where $H^2$ denotes the standard Hardy class in $\mathbb D$.
\end{definition}

Fix $z_0\in\Omega$ and let $\text{\rm orb}(\zeta_0)=\fz^{-1}(z_0)=\{\gamma(\zeta_0)\}_{\gamma\in\Gamma}$. The Blaschke product $b_{z_0}$ with zeros at
$\fz^{-1}(z_0)$ is called the Green function of the group $\Gamma$ (cf.~\cite{SY}). It is related to the standard Green
function $G(z,z_0)$ in the domain $\Omega$ by
\begin{equation*}
\log \frac{1}{|b_{z_0}(\zeta)|} = G\left(\fz(\zeta),z_0\right).
\end{equation*}
The function $b_{z_0}$ is character automorphic, that is, $b_{z_0}\circ\gamma=e^{2\pi i\mu_{z_0}}b_{z_0}$, where $\mu_{z_0}\in\Gamma^{*}$. For $b_{z_0}$ we fix the normalization $b_{z_0}(0)>0$ if $z_0\not=\infty$ and $(\fz b)(0)>0$
for the Blaschke product $b$ related to infinity.

We define $k_{\zeta_0}^{\alpha}(\zeta)=k^{\alpha}(\zeta,\zeta_0)$ as the reproducing kernel of the space $H^2(\alpha)$, that is,
\begin{equation*}
 \left\langle f, k_{\zeta_0}^{\alpha}\right\rangle = f(\zeta_0)\quad \forall f\in H^2(\alpha).
\end{equation*}
\begin{remark}\label{rema22}
Let us point out that in our case this reproducing kernels possess a representation by means of $\theta$ functions associated with the given Riemann surface \cite{Fay}. As already mentioned, $k^\alpha$ has sense in a much more general situation, say, domains of Widom type. Although, generally speaking,  they can not be represented via $\theta$ functions, they still play a role of special functions in the related problems.
\end{remark}

Let $k^{\alpha}(\zeta)=k_{0}^{\alpha}(\zeta)$, $b(\zeta)=b_{\fz(0)}(z)$, and  $\mu=\mu_{\fz(0)}$. We have an evident decomposition
\begin{equation}\label{ort1}
 H^{2}(\alpha)=\{e^{\alpha}\}\oplus b H^2(\alpha-\mu), \quad e^{\alpha}=\frac{k^{\alpha}(\zeta)}{\sqrt{k^{\alpha}(0)}}.
\end{equation}
This decomposition plays an essential role in the proof of the following theorem.

\begin{theorem}
 \label{thm:onb}
The system of functions 
\begin{equation}\label{defeal}
 e_{n}^{\alpha}(\zeta)=b^{n}(\zeta)\frac{k^{\alpha-n\mu}(\zeta)}{\sqrt{k^{\alpha -n\mu}(0)}}
\end{equation}

\begin{itemize}
 \item[(i)] forms an orthonormal basis in $H^2(\alpha)$ for $n\in \mathbb N$ and
 \item[(ii)] forms an orthonormal basis in $L^2(\alpha)$ for $n\in\mathbb Z$,
\end{itemize}
where 
\begin{equation*}
L^2(\alpha) =  \{ f \in L^2\!:~ f \circ \gamma = e^{2\pi i\alpha(\gamma)} f,~\gamma\in\Gamma \}.
\end{equation*}
\end{theorem}
\begin{proof}
 Item (i) follows from the above paragraphs and a proof for (ii) in a much more general case can be found in \cite[Theorem E]{SY}.
\end{proof}

The following theorem describes all elements of $J(E)$ for a given finite-gap set $E$.

\begin{theorem}
 \label{thm:multbyz}
The multiplication operator by $\fz$ in $L^2(\alpha)$ with respect to the basis $\{e_n^{\alpha}\}$ from Theorem~\ref{thm:onb}
is the following Jacobi matrix $J=J(\alpha)$:
\begin{equation*}
 \fz e_{n}^{\alpha}=a(n;\alpha)e_{n-1}^{\alpha} + b(n;\alpha)e_{n}^{\alpha}+a(n+1;\alpha)e^{\alpha}_{n+1},
\end{equation*}
where
\begin{equation*}
 a(n;\alpha)=\mathcal A(\alpha-n\mu), \quad\mathcal A(\alpha)=(\fz b)(0)\sqrt{\frac{k^{\alpha}(0)}{k^{\alpha+\mu}(0)}}
\end{equation*}
and 
\begin{equation*}
 b(n;\alpha)=\mathcal B(\alpha-n \mu),~~ \mathcal B(\alpha)=\frac{\fz b(0)}{b^{\prime}(0)}+
\left\{  \frac{\left(k^{\alpha}\right)^{\prime}(0)}{k^{\alpha}(0)}- \frac{\left(k^{\alpha+\mu}\right)^{\prime}(0)}{k^{\alpha+\mu}(0)}\right\}
+\frac{\left(\fz b\right)^{\prime}(0)}{b^{\prime}(0)}.
\end{equation*}
This Jacobi matrix $J(\alpha)$ belongs to $J(E)$. Thus, we have a map from $\Gamma^{*}$ to $J(E)$. Moreover, this map is one-to-one.
\end{theorem}

\begin{remark}
Using the representation of the reproducing kernels via $\theta$ functions, see Remark \ref{rema22},
 one gets $\cA(\alpha)$ and $\cB(\alpha)$ in  the form
\eqref{aal} and \eqref{bal}, respectively.
\end{remark}

\begin{remark}
The following important relation is
an immediate  consequence of the above functional model 
\begin{equation}\label{sjm}
S^{-1}J(\alpha) S=J(\alpha-\mu), \quad S e_n:= e_{n+1}.
\end{equation}
In particular, $J(\alpha)$ is periodic if and only if $N\mu=\mathbf{0}_{\Gamma^*}$ for a certain positive integer $N$.
\end{remark}

\subsection{Class $A(E,\bC)$ and Jacobi flow}

Now we turn to the functional model for $A(E,\bC)$.
The rational function $V(z)$ and the in $\bar\bbC\setminus E$ single valued function $\Delta(z)$ were defined in
\eqref{eq14}-\eqref{eq15}.
Let us list characteristic properties of $\Delta(z)$:
\begin{itemize}
 \item[(i)] $|\Delta|<1$ in $\Omega$ and $|\Delta|=1$ on $E$,
 \item[(ii)] $\Delta(\infty)=\Delta(\bc_j)=0$, $1\le j\le g$, otherwise $\Delta(z)\not=0$.
\end{itemize}
All this implies that $\Delta(z)$ is given by $\Delta(\fz(\zeta))=b(\zeta)\prod_{j=1}^g b_{\bc_j}(\zeta)$. In particular, $\mu+\sum^g_{j=1}\mu_{\bc_j}=\mathbf{0}_{\Gamma^*}$.

Let us fix $\zeta_{j}\in\bbD$ such that $\fz(\zeta_{j})=\bc_j$ and $\oc\gamma_j(\zeta_{j})=\bar \zeta_{j}$ for the generator $\oc\gamma_j$ of the group $\Gamma$. In order to construct a functional model for operators from $A(E,\bC)$, we start with the following counterpart of the orthogonal decomposition \eqref{ort1}:
\begin{equation}\label{smpbase0}
H^2(\alpha)=\{k^{\alpha}_{\zeta_{1}},\dots, k^{\alpha}_{\zeta_{g}}, k^{\alpha}\}\oplus\Delta H^2(\alpha)
=\{f^{\alpha}_0\}\oplus\dots\oplus\{f^{\alpha}_g\}\oplus\Delta H^2(\alpha),
\end{equation}
where
\begin{equation}\label{smpbase}
f_0^{\alpha}=\frac{e^{-\pi i\alpha(\oc\gamma_1)}k^\alpha_{\zeta_{1}}}{\sqrt{k^\alpha_{\zeta_{1}}(\zeta_{1})}}, \
f_1^\alpha=\frac{e^{-\pi i(\alpha-\mu_{\bc_1})(\oc\gamma_2)}b_{\bc_1} k_{\z_{2}}^{\alpha-\mu_{\bc_1}}}{\sqrt{k_{\z_{2}}^{\alpha-\mu_{\bc_1}}(\z_{2})}},...,
\ f_g^\alpha=\frac{\prod_{j=1}^g b_{\bc_j} k^{\alpha+\mu}}{\sqrt{k^{\alpha+\mu}(0)}}.
\end{equation}

\begin{theorem}
 \label{thm:onbsmp}
The system of functions 
\begin{equation}\label{gsmpbase}
f_{n}^{\alpha}=\Delta^m f_j^\alpha,\quad 
 n=(g+1)m+j, \ j\in[0,\dots,g]
\end{equation}
\begin{itemize}
 \item[(i)] forms an orthonormal basis in $H^2(\alpha)$ for $n\in \mathbb N$ and
 \item[(ii)] forms an orthonormal basis in $L^2(\alpha)$ for $n\in\mathbb Z$.
\end{itemize}
\end{theorem}
\begin{proof}
 Item (i) follows from \eqref{smpbase0} and for (ii) we have to use the description of the orthogonal complement 
 $L^2(\alpha)\ominus H^2(\alpha)$, see \cite{SY}.
\end{proof}

Similarly as we had before, this allows us to parametrize all elements of $A(E,\bC)$ for a given  $E$ by the 
characters of $\Gamma^*$.

\begin{theorem}
 \label{thm:multbyzsmp}
In the above notations
the multiplication operator by $\fz$ with respect to the basis $\{f_n^{\alpha}\}$ is a GSMP matrix $A(\alpha;\bC)\in A(E,\bC)$.
Moreover, this map $\Gamma^*\to A(E;\bC)$ is one-to-one up to the identification 
$$
(p_j,q_j)\mapsto (-p_j,-q_j) \ \text{in} \ A(E;\bC), \quad 0\le j\le g-1.
$$
\end{theorem}

\begin{proof}
The structure of the matrix is fixed by the choice of the orthonormal basis. We only need to check that, under the normalization \eqref{smpbase}, $p_j(\alpha)$
 and $q_j(\alpha)$ are real. For $\beta_j=\alpha-\sum_{k=1}^{j}\mu_{c_k}$, we have
 $$
 p_j(\alpha)=\langle \fz f_j^\alpha, f_{-1}^\alpha \rangle=(b\fz)(0)\prod_{k=1}^{j-1}b_{\bc_k}(0)
 e^{-\pi\beta_j(\oc\gamma_j)}\frac{k^{\beta_j}(0,\zeta_{j})}{\sqrt{k^{\beta_j}_{\zeta_{j}}(\zeta_{j})k^{\alpha+\mu}(0)}}.
 $$
 Since $\overline{k^\beta(\bar \zeta)}=k^\beta(\zeta)$  for all $\beta\in\Gamma^*$, we get
 $$
k^\beta(\zeta_{j})= \overline{k^\beta(\bar\z_{j})}= \overline{k^\beta(\oc\gamma_j(\zeta_{j}))}= 
 e^{-2\pi i\beta(\oc\gamma_j)}\overline{k^\beta(\zeta_{j})}.
 $$
 Therefore, $e^{-\pi i\beta(\oc\gamma_j)}\overline{k^\beta(\zeta_{j})}=
 e^{-\pi i\beta(\oc\gamma_j)}k^\beta(0,\z_{j})$ is real. Note that the square root of $e^{-2\pi i\beta(\oc\gamma_j)}$ is defined up to the multiplier $\pm 1$.
 Similarly, we prove that $q_j(\alpha)$ are real based on
 $$
 p_g(\alpha)q_j(\alpha)=\langle f^\alpha_g,f^\alpha_j \rangle.
 $$
 
 If a periodic $\oc A\in A(E,\bC)$ is given, we introduce its resolvent function $r_+(z)$, see Theorem \ref{th213} below, and define $\alpha$ exactly as in the Jacobi matrix case, see e.g. \cite{SY}.
\end{proof}

\begin{definition}\label{def:312}
We define the Jacobi flow  on $A(E;\bC)$ as the dynamical system generated by the following map: 
$$
\mathcal J A(\alpha)=A(\alpha-\mu), \quad \alpha\in\Gamma^*.
$$
\end{definition}

We can describe this operation in a very explicit form.
\begin{lemma}
 \label{thm:jacobiflowper}
Let $\cO(\alpha)$ be the unitary, periodic $(g+1)\times (g+1)$-block diagonal  matrix given by
\begin{equation}
\label{eqn:ufunctional}
 \cO(\alpha)\begin{bmatrix}e_{(g+1)m}&\dots& e_{(g+1)m+g}
 \end{bmatrix}=\begin{bmatrix}e_{(g+1)m}&\dots& e_{(g+1)m+g}
 \end{bmatrix} O(\alpha), 
 \end{equation}
 where
  \begin{equation}\label{jfp0}
O(\alpha)=
 \begin{bmatrix}
I_{g-2}&0\\
0&\begin{bmatrix}\frac{p_{g-1}(\alpha)}{\sqrt{p_{g-1}^2(\alpha)+p_g^2(\alpha)}}&\frac{p_g(\alpha)}{\sqrt{p_{g-1}^2(\alpha)+p_g^2(\alpha)}}\\ \frac{p_g(\alpha)}{\sqrt{p_{g-1}^2(\alpha)+p_g^2(\alpha)}}&\frac{-p_{g-1}(\alpha)}{\sqrt{p_{g-1}^2(\alpha)+p_g^2(\alpha)}} \end{bmatrix}
\end{bmatrix}.
 \end{equation}
 Then
\begin{equation}\label{jfp}
 \mathcal O A(\alpha;\bC):=S^{-1}\mathcal O(\alpha)^*A(\alpha;\bC)\mathcal O(\alpha)S=
 A(\alpha+\mu_{\bc_g};\bc_g,\bc_1,\dots,\bc_{g-1}).
\end{equation}
\end{lemma}

\begin{proof}
Actually, in this operation we just switched  the order of two reproducing kernels related to $\bc_g$ and $\infty$. This is a rotation in the two dimensional space. Then, up to the shift, we derived a GSMP basis of the form \eqref{smpbase}, but with the new ordering $(\bc_g,\bc_1,\dots,\bc_{g-1})$ and the new character $\alpha+\mu_{c_g}$.
\end{proof}
\begin{theorem}\label{th211}
In the above notations
\begin{equation}\label{perjac}
\cJ A(\alpha;\bC)=\cO^{\circ g} A(\alpha;\bC)
\end{equation}
\end{theorem}
\begin{proof}
We use \eqref{jfp}, having in mind that $\sum_{j=1}^g\mu_{\bc_j}=-\mu$  and that after all permutations we obtain the original ordering $\bC$.
\end{proof}

The  next lemma allows us to estimate components of the vector $f^{\alpha}_j$, $j=0,..,g,$ by means of the basis $\{e^\alpha_n\}_{n\ge 0}$.

\begin{lemma}
Let
\begin{equation}\label{defF}
f_j^\alpha=\sum_{k=0}^\infty F^j_k(\alpha) e_k^\alpha
\end{equation}
Then
\begin{equation}\label{estF}
|F^j_k(\alpha)|\le C(E) \eta^k, \quad j=0,\dots,g,
\end{equation}
where $1>\eta>\max\{|b(\bc_1)|,\dots,|b(\bc_g)|\}$.
\end{lemma}

\begin{proof}
First of all, we note that
$$
\underline C(E) \le \|k^\alpha_{\zeta_n}\|\le \overline C(E).
$$
uniformly on $\alpha\in\Gamma^*$.
Also $|b_{\bc_n}(\bc_j)|\ge \underline c(E)$ and
by definition \eqref{defeal}, $|e^\alpha_k(\bc_n)|\le \overline c(E)\eta^k$.

Since
\begin{eqnarray*}
& & e^{-\pi i\beta_j(\oc\gamma_j)}{\|k_{\zeta_j}^{\beta_j}\|}\overline{F_k^j(\alpha)}=\langle e^\alpha_k,
 \prod_{n=1}^{j-1}b_{\bc_n}
{k_{\zeta_j}^{\beta_j}}\rangle
\\
&=&\langle 
 \frac{e^\alpha_k}{ \prod_{n=1}^{j-1}b_{\bc_n}}-\sum_{n=1}^{j-1}
\frac{k^{\beta_j+\mu_{\bc_n}}_{\zeta_n}e^\alpha_k(\bc_n)}{b_{\bc_n}k^{\beta_j+\mu_{\bc_n}}_{\zeta_n}(\bc_n)
 \prod_{l=1,l\not=n}^{j-1}b_{\bc_l}(\bc_n)}
,{k_{\zeta_j}^{\beta_j}}\rangle
\\
&=& \frac{e^\alpha_k(\bc_j)}{ \prod_{n=1}^{j-1}b_{\bc_n}(\bc_j)}-\sum_{n=1}^{j-1}
\frac{k^{\beta_j+\mu_{\bc_n}(\bc_j)}_{\zeta_n}e^\alpha_k(\bc_n)}{b_{\bc_n}(\bc_j)k^{\beta_j+\mu_{\bc_n}}_{\zeta_n}(\bc_n)
 \prod_{l=1,l\not=n}^{j-1}b_{\bc_l}(\bc_n)},
\end{eqnarray*}
we get \eqref{estF}.
\end{proof}

\subsection{Transfer matrix}

In this subsection we discuss briefly the direct spectral problem of the class $A(E,\bC)$. We use the following notations
$$
s_k\vp=\begin{bmatrix} p_0\\ \vdots \\ p_{g-k} \end{bmatrix},\quad 
s_k\vq=\begin{bmatrix} q_0\\ \vdots \\ q_{g-k} \end{bmatrix}.
$$
Let $\{\delta_j\}_{j=0}^g$ be the standard basis in $\bbC^{g+1}$ and let $M_j$'s be upper triangular matrices such that
\begin{equation}\label{mdef0}
B(\vp,\vq)-\vp(\vq)^*=M(\vp,\vq):=M_0=\begin{bmatrix}M_1& 0\\
0 & 0\end{bmatrix}+(-\vp\, q_g+\vq\, p_g)\delta_g^*
\end{equation}
and
\begin{equation}\label{mdefj}
M_j=\begin{bmatrix}M_{j+1}& 0\\
0 & \bc_{g+1-j}\end{bmatrix}+(-s_j\vp\, q_{g-j}+{s_j\vq}\, p_{g-j})\delta_{g-j}^*.
\end{equation}
 for $j\ge 1$.

 \begin{theorem}\label{th213}
 Let
 \begin{equation}\label{forfa1}
\begin{bmatrix}
 R_{00}(z)& R_{0g}(z)\\
  R_{g0}(z)& R_{gg}(z)
\end{bmatrix}=
\begin{bmatrix}
\langle(B_0-z)^{-1}\vp,\vp\rangle& \langle(B_0-z)^{-1}\delta_g,\vp\rangle \\
 \langle(B_0-z)^{-1}\vp, \delta_g\rangle& \langle(B_0-z)^{-1}\delta_g,\delta_g\rangle
\end{bmatrix}
\end{equation}
and
$$
r_+(z)=\|\vp\|^2\langle(A_+-z)^{-1}\tilde e_0, \tilde e_0\rangle.
$$
Then the shift by one block for  one-sided GSMP matrix $A_+\mapsto A_+^{(1)}$, see \eqref{gsmpshift}, by means of the spectral function has the following form
\begin{equation}\label{forfa3}
r_+(z)=\frac{\fA_{11}(z) r^{(1)}_+(z)+\fA_{12}(z)}{\fA_{21}(z)r^{(1)}_+(z)+\fA_{22}(z)},
\end{equation}
where
\begin{equation}\label{forfa2}
\fA(z):=\begin{bmatrix}
\fA_{11}&\fA_{12}\\
\fA_{21}&\fA_{22}
\end{bmatrix}(z)=
\frac{1}{R_{0g}(z)}\begin{bmatrix}
R_{00}R_{gg}-R_{0g}^2& -R_{00}\\
R_{gg}&-1
 \end{bmatrix}(z).
\end{equation}
 \end{theorem}
\begin{proof}
We represent $A_+$ as a two dimensional perturbation of the block diagonal matrix
\begin{equation}\label{gsmpshift}
A_+=\begin{bmatrix}
B(\vp,\vq)&0\\ 0 &A_+^{(1)}
\end{bmatrix}+\|\vp\,^{(1)}\|(e_g\langle \cdot, \tilde e^{(1)}_0 \rangle+\tilde e^{(1)}_0\langle \cdot, e_g \rangle)
\end{equation}
and apply the resolvent perturbation formula.
\end{proof}
Note that in the definition \eqref{forfa2} we use the normalization $\det\fA(z)=1$.
\begin{theorem}
Let
\begin{equation}\label{thematrix}
\fa(z;p,q)=\fa(z,\infty;p,q)=
\begin{bmatrix}
0&-{p}\\
\frac 1{p}&\frac{z-pq}{p}
\end{bmatrix},
\end{equation}
and
\begin{equation}\label{facto2}
\fa(z,\bc;p,q)=I-\frac{1}{\bc-z}\begin{bmatrix}
p\\ q
\end{bmatrix}\begin{bmatrix}
p&q
\end{bmatrix}\fj,\quad
\fj=
\begin{bmatrix}
0&-1\\
1&0
\end{bmatrix}.
\end{equation}
be the so-called Blaschke-Potapov factors \cite{POT}. Then
\begin{equation}\label{facto3}
\fA(z)=\fa(z,\bc_1;p_0,q_0)\fa(z,\bc_2;p_1,q_1)\dots\fa(z,\bc_g;p_{g-1},q_{g-1})\fa(z;p_g,q_g).
\end{equation}
\end{theorem}

\begin{proof} We use the representation \eqref{mdef0} and definitions \eqref{forfa1}, \eqref{thematrix} to get
$$
\fA(z)=\fA_0(z)\fa(z;p_g,q_g)
$$
where
\begin{equation}\label{facto1}
\fA_0(z)=I-\begin{bmatrix} \langle (M_1-z)^{-1} {s_1\vp},{s_1\vp}\rangle&
\langle (M_1-z)^{-1}{s_1\vq},{s_1\vp}\rangle\\
\langle (M_1-z)^{-1}{s_1\vp},{s_1\vq}\rangle & \langle (M_1-z)^{-1}{s_1\vq},{s_1\vq}\rangle
\end{bmatrix}\fj.
\end{equation}
Then, we use one after another \eqref{mdefj} and definitions \eqref{facto2} to get 
$$
\fA_{j-1}(z)=\fA_{j}(z)\fa(z,\bc_{g+1-j};p_{g-j},q_{g-j}),
$$
where
$$
\fA_{j-1}(z)=I-\begin{bmatrix} \langle (M_j-z)^{-1}{s_j\vp},{s_j\vq}\rangle&
\langle (M_j-z)^{-1}{s_j\vq},{s_j\vp}\rangle\\
\langle (M_j-z)^{-1}{s_j\vp},{s_j\vq}\rangle & \langle (M_j-z)^{-1}{s_j\vq},{s_j\vq}\rangle
\end{bmatrix}\fj.
$$
That is, we obtain \eqref{facto3}.
\end{proof}



\begin{definition} Let 
$A\in A(E,\bC)$.
Then the product \eqref{facto3}
is called the \textit{transfer matrix} associated  with the given $A$.
\end{definition}

The role of the transfer matrix is described in the following theorem.

\begin{theorem}\label{inth7}
Let $A\in A(E,\bC)$  with the transfer matrix $\fA(z)$, given in \eqref{facto3}, and
let $V(z):=\tr\, \fA(z)$. Then the spectrum $E$ of $A$ is given by 
\begin{equation}\label{spectrum}
E=V^{-1}([-2,2])=\{x:\, V(x)\in[-2,2]\}.
\end{equation}
Moreover,
\begin{equation}\label{vofa}
V(z)=\lambda_0 z+\bc_0+\sum_{j=1}^g\frac{\lambda_j}{\bc_j-z},\quad \lambda_j>0,
\end{equation}
where
\begin{equation}\label{explpqg}
\lambda_0 p_g=1, \quad {\lambda_0}\sum_{j=0}^{g}p_j q_j+{\bc_0}=0,
\end{equation}
and
\begin{eqnarray}\label{explrhk}
\lambda_k&=&-\Res_{\bc_k}\tr \fA(z)\nonumber\\
&=&
\begin{bmatrix}
-q_{k-1}&p_{k-1}
\end{bmatrix}
\prod_{j=k}^{g-1}\fa(\bc_k,\bc_{j+1};p_j,q_j)\fa(\bc_k;p_g,q_g)
\nonumber \\ 
&\times&
\prod_{j=0}^{k-2}\fa(\bc_k,\bc_{j+1};p_j,q_j)\begin{bmatrix}
p_{k-1}\\ q_{k-1}
\end{bmatrix}.
\end{eqnarray}
\end{theorem}

\begin{proof}
A proof of \eqref{vofa} is the same as for the transfer matrix in the Jacobi matrices case. The relations
\eqref{explpqg} and \eqref{explrhk} follow immediately from \eqref{facto3}.
\end{proof}

\begin{proof}[Proof of Proposition \ref{prop18}]
First of all, we have a parametrization of $A(E,\bC)$ by the characters $\Gamma^*$.
It is evident that, in the basis \eqref{gsmpbase}, multiplication by $\Delta$  is the shift $S^{g+1}$, $\Delta f_n^{\alpha}=f^\alpha_{n+(g+1)}$. Thus, the magic formula for GSMP matrices corresponds to the definition \eqref{eq15}. 
\end{proof}

Later we will use another representation for $q_g$.

\begin{lemma}\label{lem:c0formula}
$q_g$ allows the following alternative representation 
\begin{eqnarray}
q_g+\bc_0
&=&-\sum_{k=1}^{g}\tr \{\prod_{j=0}^{k-2}\fa(\bc_k,\bc_{j+1};p_j,q_j)\begin{bmatrix}
p_{k-1}\\ q_{k-1}
\end{bmatrix}
\begin{bmatrix}
-q_{k-1}& p_{k-1}
\end{bmatrix}\nonumber
\\
&\times& \prod_{j=k}^{g-1}\fa(\bc_k,\bc_{j+1};p_j,q_j)\begin{bmatrix}
0& 0\\ 0&\frac 1 {p_{g}}
\end{bmatrix}\}\label{alternativqg}
\end{eqnarray}
\end{lemma}
\begin{proof}
From the second relation in \eqref{explpqg} one has
$$
q_g+\bc_0=\frac 1{2\pi i}\oint_{|z|=R}\tr\, \fA_0(z)\begin{bmatrix}
0& 0\\ 0&\frac 1{p_g}
\end{bmatrix}dz=\sum_{k=1}^g\Res_{\bc_k}
\tr\, \fA_0(z)\begin{bmatrix}
0& 0\\ 0&\frac 1{p_g}
\end{bmatrix},
$$
which is \eqref{alternativqg}.
\end{proof}

\section{GSMP matrices, general case.}
 
 We hope after Theorem \ref{inth7}, and especially \eqref{explrhk}, it would be easy to perceive the following notations. 
\begin{notation} For $k=1,\dots, g$ the functions (polynomials)
$$
\Lambda^\#_{j,k}=\Lambda^\#_k(p^{(j+1)}_0,q^{(j+1)}_0,\dots,
p^{(j+1)}_{k-1},q^{(j+1)}_{k-1};p^{(j)}_{k-1},q^{(j)}_{k-1},\dots, p^{(j)}_g,q^{(j)}_g)
$$ 
are given by 
\begin{eqnarray}\label{deflad}
\Lambda^\#_{j,k}&=&-\tr\{
\prod_{m=0}^{k-2}\fa(\bc_k,\bc_{m+1};p^{(j+1)}_m,q^{(j+1)}_m)
\begin{bmatrix}
p^{(j+1)}_{k-1}\\ q^{(j+1)}_{k-1}
\end{bmatrix}
\begin{bmatrix}
p^{(j)}_{k-1}& q^{(j)}_{k-1}
\end{bmatrix}\fj \nonumber\\
&\times&
\prod_{m=k}^{g-1}\fa(\bc_k,\bc_{m+1};p^{(j)}_m,q^{(j)}_m)
\fa(\bc_k;p^{(j)}_g,q^{(j)}_g)\}
\end{eqnarray}
If $p^{(j)}_m=p^{(j+1)}_m=p_m$  and $q^{(j)}_m=q^{(j+1)}_m=q_m$ for all $m\in [0,g]$ we simplify this notations to
\begin{equation}\label{defla}
\Lambda_k(\vp,\vq)=\Lambda^\#_k(p_0,q_0,\dots,
p_{k-1},q_{k-1};p_{k-1},q_{k-1},\dots, p_g,q_g),
\end{equation}
and $\Lambda^\#_{k}:=\Lambda^\#_{j,k}$.
\end{notation}
\begin{lemma}\label{lem:gsmpEntries}
Let $A$ belong to the GSMP class. Then the vector $f(\bc_k):=(\bc_k-A)^{-1}e_{k-1}$, $k=1,..,g$, is of the form
\begin{equation}\label{egsmp21}
f_j(\bc_k)=0, \quad j\not\in\{-1,0,1\}, \ f_j(\bc_k)=\{(f_j(\bc_k))_n\}_{n=0}^g\in\bbC^{g+1}
\end{equation}
and $(f_{-1})_0=...=(f_{-1})_{k-2}=0$, $(f_{1})_k=...=(f_{1})_{g}=0$. The non-zero entries are given by:
\begin{eqnarray}
(f_{-1})_{k-1}&=&\frac{1}{\Lambda^\#_{-1,k}},\qquad (f_{1})_{k-1}=\frac{1}{\Lambda^\#_{0,k}}, \label{eqsmp21a}\\
(f_{-1})_l&=&-\frac{1}{\bc_{l+1}-\bc_k}\begin{bmatrix}
q^{(-1)}_{k-1}& -p^{(-1)}_{k-1}
\end{bmatrix}\nonumber\\
&\times&
\prod_{j=k}^{l-1}\fa(\bc_k,\bc_{j+1},p_j^{(-1)},q_j^{(-1)})
\begin{bmatrix}
p_l^{(-1)}\\
q_l^{(-1)}
\end{bmatrix}
\frac{1}{\Lambda^\#_{-1,k}}, \label{eqsmp21b}\\
(f_{1})_{m}&=&-\frac{1}{\bc_{m+1}-\bc_k}\begin{bmatrix}
q^{(1)}_{m}& -p^{(1)}_{m}
\end{bmatrix} \nonumber\\
&\times&\prod_{j=m+1}^{k-2}\fa(\bc_k,\bc_{j+1},p_j^{(1)},q_j^{(1)})
\begin{bmatrix}
p_{k-1}^{(1)}\\
q_{k-1}^{(1)}
\end{bmatrix}
\frac{1}{\Lambda^\#_{0,k}}.\label{eqsmp21d}
\end{eqnarray}
\end{lemma}
\begin{proof}
Using \eqref{mdef0}, we solve the following linear system 
\begin{eqnarray}
-A_{-1}f_{-1}&=&0,  \label{egsmp22}\\
(\bc_k-B_{-1})f_{-1}-A_0f_0&=&0, \label{egsmp23}\\
-A_0^*f_{-1}+(\bc_k-B_{0})f_{0}-A_1f_1&=&\delta_{k-1},\label{egsmp24}\\
-A_1^*f_{0}+(\bc_k-B_{1})f_{1}&=&0,\label{egsmp25}\\
-A_2^*f_{1}&=&0.\label{egsmp26}
\end{eqnarray}
From this we obtain \eqref{eqsmp21a}, and the following recurrence relations:
\begin{eqnarray*}
(f_{-1})_l=\frac{1}{\bc_{l+1}-\bc_k}
\Bigg(
\begin{bmatrix}
-q^{(-1)}_{k-1}&p^{(-1)}_{k-1}
\end{bmatrix}
\begin{bmatrix}
p_l^{(-1)}\\
q_l^{(-1)}
\end{bmatrix}
\frac{1}{\Lambda^\#_{-1,k}}\\
-\sum_{j=k}^{l-1}
\begin{bmatrix}
q_j^{(-1)}& -p_j^{(-1)}
\end{bmatrix}
\begin{bmatrix}
p_l^{(-1)}\\
q_l^{(-1)}
\end{bmatrix}
(f^k_{-1})_{j}
\Bigg)&
\end{eqnarray*}
and 
\begin{eqnarray*}
(f_{0})_m=-\frac{1}{\bc_{m+1}-\bc_k}
\Bigg(
\begin{bmatrix}
q^{(1)}_{m}& -p^{(1)}_{m}
\end{bmatrix}
\begin{bmatrix}
p_{k-1}^{(1)}\\
q_{k-1}^{(1)}
\end{bmatrix}
\frac{1}{\Lambda^\#_{0,k}}&\\
-\sum_{j=m+1}^{k-2}
\begin{bmatrix}
	q^{(1)}_{m}& -p^{(1)}_{m}
\end{bmatrix}
\begin{bmatrix}
	-p_{j}^{(1)}\\
	-q_{j}^{(1)}
\end{bmatrix}
(f^k_{0})_{j}
\Bigg).&
\end{eqnarray*}
By induction, we see that 
\begin{eqnarray*}
(f_{-1})_l&=&
\frac{1}{\bc_{m+1}-\bc_k}
\begin{bmatrix}
-q^{(-1)}_{k-1}& p^{(-1)}_{k-1}
\end{bmatrix}\\
&\times&\Bigg(
I
-\sum_{j=k}^{l-1}\frac{1}{\bc_{j+1}-\bc_k}
\prod_{i=k}^{j-1}\fa(\bc_k,\bc_{i+1},p_i^{(-1)},q_i^{(-1)})
\begin{bmatrix}
p_j^{(-1)}\\
q_j^{(-1)}
\end{bmatrix}
\begin{bmatrix}
p_j^{(-1)}& q_j^{(-1)}
\end{bmatrix}\fj
\Bigg)\\
&\times&\begin{bmatrix}
p_l^{(-1)}\\
q_l^{(-1)}
\end{bmatrix}
\frac{1}{\Lambda^\#_{-1,k}}.
\end{eqnarray*}
Recursively, we find
\begin{eqnarray*}
& I&
-\sum_{j=k}^{l-1}\frac{1}{\bc_{j+1}-\bc_k}
\prod_{i=k}^{j-1}\fa(\bc_k,\bc_{i+1},p_i^{(-1)},q_i^{(-1)})
\begin{bmatrix}
p_j^{(-1)}\\
q_j^{(-1)}
\end{bmatrix}
\begin{bmatrix}
p_j^{(-1)}& q_j^{(-1)}
\end{bmatrix}\fj\\
&=&\fa(\bc_k,\bc_{k+1},p^{(-1)}_{k},q^{(-1)}_{k})\\
&\times&\Bigg(I-\sum_{j=k+1}^{l-1}\frac{1}{\bc_{j+1}-\bc_k}
\prod_{i=k+1}^{j-1}\fa(\bc_k,\bc_{i+1},p_i^{(-1)},q_i^{(-1)})
\begin{bmatrix}
p_j^{(-1)}\\
q_j^{(-1)}
\end{bmatrix}
\begin{bmatrix}
p_j^{(-1)}& q_j^{(-1)}
\end{bmatrix}\fj\Bigg)\\
&=&\prod_{j=k}^{l-1}\fa(\bc_k,\bc_{j+1},p_j^{(-1)},q_j^{(-1)}).
\end{eqnarray*}
In a similar way, one can show \eqref{eqsmp21d}.
\end{proof}

\begin{theorem}\label{defaltdef}
A GSMP structured matrix $A$ belongs to the GSMP class if and only if the forming sequences $\{\vp_j,\vq_j\}$ satisfy the following conditions
\begin{equation}\label{altdef}
\inf_{j\in\bbZ}\Lambda^\#_{j,k}>0, \ \text{for all}\ k=1,\dots,g.
\end{equation}
\end{theorem}
\begin{proof}\label{thstr} Solvability of the system \eqref{egsmp22}-\eqref{egsmp26} is equivalent to \eqref{eqsmp21a} and \eqref{eqsmp21c}. In this case all $\bc_k-A$ are invertible.

\end{proof}
In particular, we can evaluate the \textit{diagonal} entries of $V(A)$, 
an explicit formula for which is required in what follows.
\begin{lemma}
Let
\begin{equation}\label{altda1}
\begin{bmatrix}
\tilde \pi ^{(0)}_{k-1}\\
\tilde \rho^{(0)}_{k-1}
\end{bmatrix}=\prod_{j=0}^{k-2}\fa(\bc_k,\bc_{j+1};p^{(0)}_j,q_j^{(0)})\begin{bmatrix}p_{k-1}^{(0)}\\
q_{k-1}^{(0)}
\end{bmatrix}
\end{equation}
and
\begin{equation}\label{altda2}
\begin{bmatrix}
\pi ^{(-1)}_{k-1}&
 \rho^{(-1)}_{k-1}
\end{bmatrix}=
\begin{bmatrix}q_{k-1}^{(-1)}&
-p_{k-1}^{(-1)}
\end{bmatrix}
\prod_{j=k}^{g-1}\fa(\bc_k,\bc_{j+1};p^{(-1)}_j,q_j^{(-1)})~\fj
\end{equation}
Then
\begin{equation}\label{altd2bis}
\langle V(A) e_{-1},e_{-1}\rangle=\lambda_0 p^{(-1)}_g q^{(-1)}_g+\bc_0-\sum_{k=1}^g\lambda_k\frac{\pi^{(-1)}_{k-1}\tilde \rho^{(0)}_{k-1}}{p_g^{(-1)}\Lambda^\#_{-1,k}}.
\end{equation}

\end{lemma}

\begin{proof}
Indeed, we have to find $h=h(\bc_k)=(\bc_k-A)^{-1}e_{-1}$, $1\le k\le g$.
We note that $h$ has only two non-trivial components $h_{-1}$ and $h_0$ for which we have
$$
\begin{matrix}
-A_{-1}h_{-1}&=&0,\\
(\bc_k-B_{-1}) h_{-1}-A_0 h_0&=&\delta_g,\\
-A_0^* h_{-1}+(\bc_k-B_0)h_0&=&0,\\
-A_1^*h_0&=&0.
\end{matrix}
$$
In notations \eqref{altda1}, \eqref{altda2} we obtain 
$$
\begin{bmatrix}
-\pi^{(-1)}_{k-1}& 0&p_g^{(-1)} \\
0&\tilde \rho^{(0)}_{k-1} &1 \\
p_g^{(-1)}\rho^{(-1)}_{k-1}& -\tilde \pi_{k-1}^{(0)}&-(p_g^{(-1)}q_g^{(-1)} -\bc_k)\\
\end{bmatrix}
\begin{bmatrix}
(h_{-1})_{k-1}\\
(h_{0})_{k-1}\\
(h_{-1})_g
\end{bmatrix}=
\begin{bmatrix}
0\\
0\\
1
\end{bmatrix}.
$$
Note that the determinant of the matrix of the given system is of the form
$$
\pi^{(-1)}_{k-1}\tilde \rho^{(0)}_{k-1}(p^{(-1)}_{g}q_g^{(-1)}-\bc_k)-\rho^{(-1)}_{k-1}\tilde \rho^{(0)}_{k-1} (p_g^{(-1)})^2-
\pi^{(-1)}_{k-1}\tilde \pi^{(0)}_{k-1}=p_g^{(-1)}\Lambda^\#_{-1,k}.
$$
Thus, we get 
\begin{equation*}\label{altd2}
(h_{-1})_g=-\frac{\pi^{(-1)}_{k-1}\tilde \rho^{(0)}_{k-1}}{p_g^{(-1)}\Lambda^\#_{-1,k}},
\end{equation*}
which implies \eqref{altd2bis}.
Moreover, we have
\begin{equation*}
(h_{-1})_{k-1}=-\frac{\tilde\rho^{(0)}_{k-1}}{\Lambda^\#_{-1,k}}\quad\text{and}\quad
(h_{0})_{k-1}=\frac{\pi^{(-1)}_{k-1}}{p_g^{(-1)}\Lambda^\#_{-1,k}}.
\end{equation*}
\end{proof}
\begin{remark}
Note that in the periodic case \eqref{altd2bis} leads to \eqref{alternativqg}.
\end{remark}
\section{Jacobi flow, general case}
Let us mention once again that Theorem \ref{th211} gives already a certain hint for the correct definition of the Jacobi flow. It will be  defined via the unitary transformation, which after $g$ rotations and one shift, maps GSMP$(\bc_1,\dots,\bc_g)$ 
into itself.
The first rotation creates the matrix $\tilde A$ which belongs (up to a suitable shift) to GSMP$(\bc_g,\bc_1,...,\bc_{g-1})$ class. Then we create a matrix of the class GSMP$(\bc_{g-1},\bc_g,\bc_1,...,\bc_{g-2})$, and so on... On the last step (making the shift) we get the required  Jacobi flow transform, see \eqref{jflow}.
Having in  mind \eqref{eqn:ufunctional} and \eqref{jfp0}, we give the following definition.

\begin{definition}\label{defo}
We  define the  map 
$$
\cO:\text{GSMP}(\bc_1,\bc_2,...,\bc_{g})\to\text{GSMP}(\bc_g,\bc_1,...,\bc_{g-1})
$$ in the following way. Let 
$O=O_A$ be the block-diagonal matrix
$$
O=\begin{bmatrix} \ddots& & & \\ & O_{-1}& &\\
& & O_{0} & \\
& & &\ddots
\end{bmatrix}
$$
where $O_k$ are the $(g+1)\times(g+1)$ orthogonal matrices
$$
O_k=\begin{bmatrix}I_{g-2}& 0\\
0& \begin{bmatrix}
\sin\phi_k&
\cos\phi_k\\
\cos\phi_k&-\sin\phi_k
\end{bmatrix}
\end{bmatrix}, 
$$
and
$$
\begin{bmatrix}
\sin\phi_k&
\cos\phi_k\\
\cos\phi_k&-\sin\phi_k
\end{bmatrix}
=\frac{
\begin{bmatrix}
p^{(k)}_{g-1}&
p^{(k)}_{g}\\
p^{(k)}_{g}&-p^{(k)}_{g-1}
\end{bmatrix}}{\sqrt{(p^{(k)}_{g-1})^2+(p^{(k)}_{g})^2}}.
$$
Then
\begin{equation}\label{varpidef}
\cO A:=S O_A^* AO_A S^{-1}.
\end{equation}

\end{definition}

\begin{proof}[Correctness of Definition \ref{defo}]
Let us demonstrate the correctness of this definition.
 For $p$-entries of $\tilde A=\cO A$ we get
\begin{equation}\label{def3}
\tilde p^{(0)}_j=p^{(0)}_{j-1}\cos\phi_{-1}, \quad 1\le j\le g-1;
\end{equation}
\begin{equation}\label{def2}
\tilde p^{(0)}_g=\sqrt{(p^{(0)}_{g-1})^2+(p^{(0)}_{g})^2}\cos\phi_{-1}=
\sqrt{\frac{(p^{(0)}_{g-1})^2+(p^{(0)}_{g})^2}{(p^{(-1)}_{g-1})^2+(p^{(-1)}_{g})^2}}p^{(-1)}_g.
\end{equation}
Also,
\begin{eqnarray}\label{def0}
&\begin{bmatrix}\tilde q_g^{(-1)}\tilde p_g^{(-1)}& \tilde p^{(0)}_0\\
\tilde p_0^{(0)}& \tilde q_0^{(0)}\tilde p_0^{(0)}\end{bmatrix}+
\begin{bmatrix} 0&0\\
0&\bc_g\end{bmatrix}=&
\\
\nonumber
&\begin{bmatrix}
\sin\phi_{-1}&
\cos\phi_{-1}\\
\cos\phi_{-1}&-\sin\phi_{-1}
\end{bmatrix}
\left\{
\begin{bmatrix}\bc_g&0\\
0&0\end{bmatrix}+
\begin{bmatrix}q^{(-1)}_{g-1}p^{(-1)}_{g-1}&q^{(-1)}_{g-1}p^{(-1)}_{g}\\
q^{(-1)}_{g-1}p^{(-1)}_{g}&q^{(-1)}_{g} p^{(-1)}_{g}
\end{bmatrix}\right\}
\begin{bmatrix}
\sin\phi_{-1}&
\cos\phi_{-1}\\
\cos\phi_{-1}&-\sin\phi_{-1}
\end{bmatrix}.&
\end{eqnarray}

For the $q$-entries we have
\begin{equation}\label{def1}
\tilde q^{(0)}_0\tilde p^{(0)}_g=-\sin\phi_{-1}\sqrt{(p^{(0)}_{g-1})^2+(p^{(0)}_{g})^2}
\end{equation}
and
\begin{equation}\label{def4}
\tilde q_{j}^{(0)}\tilde p^0_g=q_{j-1}^{(0)} \sqrt{(p^{(0)}_{g-1})^2+(p_g^{(0)})^2}, \quad 1\le j\le g-1.
\end{equation}

Now we have to check the compatibility condition for $\tilde q_0^0$ because it appears twice in \eqref{def0} and \eqref{def1}. In other words, we have to check that
\begin{equation}\label{comp}
\left|\begin{matrix} \tilde p_0^{(0)}& \tilde p_g^{(0)}\\ \tilde  q_0^{(0)} \tilde p_0^{(0)} &\tilde q_{0}^{(0)}\tilde p_g^{(0)} \end{matrix}\right|=0
\end{equation}
holds,  if the first column is taken from \eqref{def0} and the second one is formed by \eqref{def1} and \eqref{def2}.
We use
$$
\begin{bmatrix}
\sin\phi_{-1}&
\cos\phi_{-1}\\
\cos\phi_{-1}&-\sin\phi_{-1}
\end{bmatrix}\begin{bmatrix} \bc_g&0\\
0&0\end{bmatrix}
\begin{bmatrix}
\sin\phi_{-1}&
\cos\phi_{-1}\\
\cos\phi_{-1}&-\sin\phi_{-1}
\end{bmatrix}
$$
$$
=\bc_g-\begin{bmatrix}
\sin\phi_{-1}&
\cos\phi_{-1}\\
\cos\phi_{-1}&-\sin\phi_{-1}
\end{bmatrix}\begin{bmatrix} 0&0\\
0&\bc_g\end{bmatrix}
\begin{bmatrix}
\sin\phi_{-1}&
\cos\phi_{-1}\\
\cos\phi_{-1}&-\sin\phi_{-1}
\end{bmatrix}
$$
to get
$$
\begin{bmatrix} \tilde p_0^{(0)}\\ \tilde  q_0^{(0)} \tilde p_0^{(0)} \end{bmatrix}=
\begin{bmatrix}\cos\phi_{-1}\\
-\sin\phi_{-1}
\end{bmatrix}\sin\phi_{-1}\bc_g
+\begin{bmatrix}
\cos\phi_{-1}\\
-\sin\phi_{-1}
\end{bmatrix}p_g^{(-1)}(q_{g-1}^{(-1)}\cos\phi_{-1}-q_{g}^{(-1)}\sin\phi_{-1}).
$$
But
$$
\begin{bmatrix}  \tilde p_g^{(0)}\\ \tilde q_{0}^{(0)}\tilde p_g^{(0)} \end{bmatrix}=\begin{bmatrix}\cos\phi_{-1}\\
-\sin\phi_{-1}
\end{bmatrix}\sqrt{(p^{(0)}_{g-1})^2+(p^{(0)}_{g})^2}.
$$
Thus, \eqref{comp} is proved.

Using \eqref{def3}, \eqref{def4} and \eqref{def2}, we obtain
$$
\tilde q_{j}^{(0)}\tilde p_{k}^{(0)}=q_{j-1}^{(0)}\frac{\sqrt{(p^{(0)}_{g-1})^2+(p^{(0)}_{g})^2}}{\tilde p^{(0)}_g}p^{(0)}_{k-1}\cos\phi_{-1}=
q_{j-1}^{(0)}p_{k-1}^{(0)}, \quad 1\le j,k\le g-1.
$$
Moreover, by  \eqref{def3}, \eqref{def2} and \eqref{def1}, we see that
$$
\tilde q^{(0)}_0\tilde p^{(0)}_j=-\sin\phi_{-1}\frac{\sqrt{(p^{(0)}_{g-1})^2+(p^{0}_{g})^2}}{\tilde p^{(0)}_g}p^0_{j-1}\cos\phi_{-1}
=-\sin\phi_{-1}p^{(0)}_{j-1}
$$
and, therefore, the GSMP$(\bc_g,\bc_1,...,\bc_{g-1})$-structure of $\tilde A$ is completely established.
\end{proof}

Our next definition is a counterpart of \eqref{perjac}.

\begin{definition}\label{defjf}
We  define the  Jacobi flow transform 
$$
\cJ:\text{GSMP}(\bc_1,\bc_2,...,\bc_{g})\to\text{GSMP}(\bc_1,\bc_2,...,\bc_{g})
$$ by 
\begin{equation}\label{jflow}
\cJ A=S^{-(g+1)}\cO^{\circ g} AS^{g+1}=\cO^{\circ g}(S^{-(g+1)} AS^{g+1}).
\end{equation}
\end{definition}

Let us  note that 
\begin{equation}\label{comvarpi}
S^{-(g+1)}\cO(A) S^{g+1}=\cO(S^{-(g+1)}A S^{g+1}).
\end{equation}
This has an important consequence.

\begin{corollary}\label{corshift}
\begin{equation}\label{shiftjflow}
\cO(\cJ^{\circ n} A)=\cJ^{\circ n}(\cO A).
\end{equation}
\end{corollary}
\begin{proof} Due to \eqref{jflow} and \eqref{comvarpi} we get
$$
\cJ (\cO A)=\cO^{\circ g}(S^{-(g+1)} \cO AS^{g+1})=\cO^{\circ (g+1)}(S^{-(g+1)} AS^{g+1})=\cO(\cJ A).
$$
\end{proof}

Let us turn to explicit formulas for the given transform. First of all, we note that
\begin{equation}\label{jfex0}
\cJ A=S^{-1}U^*_A A U_A S,
\end{equation}
where  $U_A$ is a $(g+1)\times (g+1)$-block diagonal matrix
$$
U_A=U=\begin{bmatrix} \ddots& & & \\ & U(\vp_{-1})& &\\
& & U(\vp_0) & \\
& & &\ddots
\end{bmatrix}.
$$
The block matrices $U=U(\vp)$ are given by products of orthogonal matrices, i.e.,
\begin{equation}\label{eq412}
U(\vp)=\begin{bmatrix}I_{g-2}& 0\\
0& \begin{bmatrix}
\sin\phi_g&
\cos\phi_g\\
\cos\phi_g&-\sin\phi_g
\end{bmatrix}
\end{bmatrix}
\dots
\begin{bmatrix}
 \begin{bmatrix}
\sin\phi_1&
\cos\phi_1\\
\cos\phi_1&-\sin\phi_1
\end{bmatrix} & 0\\
0& I_{g-2}
\end{bmatrix}
\end{equation}
where
\begin{eqnarray}
\begin{bmatrix}
\sin\phi_g&\cos\phi_g
\end{bmatrix}&=&\frac{\begin{bmatrix}p_{g-1}&p_g
\end{bmatrix}}{\sqrt{p_{g-1}^2+p_{g}^2}}\label{ang21}\\
\begin{bmatrix}
\sin\phi_{g-1}&\cos\phi_{g-1}
\end{bmatrix}&=&\frac{\begin{bmatrix}p_{g-2}&\sqrt{p_{g-1}^2+p_g^2}
\end{bmatrix}}{\sqrt{p_{g-2}^2+p_{g-1}^2+p_{g}^2}}\\
\begin{matrix}\vdots& & \vdots&  &\vdots
\end{matrix}
&=& \begin{matrix}\vdots & &\vdots & & \vdots
\end{matrix}\nonumber\\
\begin{bmatrix}
\sin\phi_{1}&\cos\phi_{1}
\end{bmatrix}&=&\frac{\begin{bmatrix}p_{0}&\sqrt{p_{1}^2+\dots+p_g^2}
\end{bmatrix}}{\sqrt{p_{0}^2+p_{1}^2+\dots+p_{g}^2}}\label{ang24}
\end{eqnarray}

\begin{theorem}\label{th54}
Let $A(1)=\cJ A$ and let $\{p^{(j)}_k(1),  q_k^{(j)}(1)\}$ be generating coefficient sequences of $A(1)$. Then
\begin{eqnarray}\label{jfex}
\begin{bmatrix}
 q_0^{(j)}(1)\\
\vdots\\
q_{g-1}^{(j)}(1)
\end{bmatrix}
&=& 
\|\vp_{j}\|
\begin{bmatrix}
\frac{-p_0^{(j)}}{\sqrt{(p_0^{(j)})^2+\dots+(p_g^{(j)})^2}{\sqrt{(p_1^{(j)})^2+\dots+(p_g^{(j)})^2}}}\\
\vdots\\
\frac{-p_{g-1}^{(j)}}{\sqrt{(p_{g-1}^{(j)})^2+(p_g^{(j)})^2}p_g^{(j)}}
\end{bmatrix}
\\ \label{jfex1}
q_g^{(j)}(1)
&=&
\frac 1{p_g^{(j)}}
\frac{\langle B(\vp_{j+1},\vq_{j+1})\,\vp_{j+1},  \vp_{j+1}\rangle\|\vp_{j}\|}{\|\vp_{j+1}\|^3 }
\\\label{jfex3}
\begin{bmatrix}
*\\  p_0^{(j)}(1)\\ \vdots \\ p_{g-1}^{(j)}(1)
\end{bmatrix}
&=&
U^*(\vp_j)B(\vp_{j},\vq_j)U(\vp_j)\delta_0
\\ \label{jfex2}
p_g^{(j)}(1)
&=&
\frac{\|\vp_{j+1}\|}{\|\vp_{j}\|}p_g^{(j)}
\end{eqnarray}
\end{theorem}

\begin{proof}
According to Definition \ref{defjf}, we have to introduce the angles \eqref{ang21}-\eqref{ang24} and
their product \eqref{eq412}.
In this case
$$
\begin{bmatrix}
p^{(0)}_0&\dots&p^{(0)}_g
\end{bmatrix} U(\vp_{0})=\begin{bmatrix}
1&0&\dots&0
\end{bmatrix}{\sqrt{(p^{(0)}_{0})^2+(p^{(0)}_{1})^2+\dots+(p^{(0)}_{g})^2}},
$$
and
$$
U^*(\vp_{-1})\begin{bmatrix} 0\\
\vdots\\ 0\\ 1
\end{bmatrix}=
\begin{bmatrix} 
\cos\phi^{-1}_1\cos \phi^{-1}_2\dots\cos\phi^{-1}_g\\\
-\sin\phi^{-1}_1\cos \phi^{-1}_2\dots\cos\phi^{-1}_g\\ \hdots\\ -\sin\phi^{-1}_g
\end{bmatrix}.
$$
Since
$$
\begin{bmatrix}
 p_g^{(-1)}(1)\\
p_g^{(-1)}(1) q_0^{(-1)}(1)\\
\vdots\\
 p_g^{(-1)}(1) q_{g-1}^{(-1)}(1)
\end{bmatrix}=U^*(\vp_{-1})A_0U(\vp_{0})
\begin{bmatrix}
1\\0\\ \vdots\\0
\end{bmatrix}
$$
the combination of the two above displayed identities implies \eqref{jfex}.
\end{proof}

\begin{remark}
Let
\begin{equation}\label{remjf}
a(0):=\|\vp_0\|, \quad  b(-1):=\frac{\langle B(\vp_0,\vq_0) \vp_0,\vp_0 \rangle}{\|\vp_0\|^2}.
\end{equation}
We would like to point out that in the Jacobi flow transform \eqref{jfex}-\eqref{jfex2} the coefficients of the blocks 
$A(\vp_0)$, $B(\vp_0,\vq_0)$ are involved only via these two functions of the vectors $\vq_0,\vp_0$. Indeed, in \eqref{jfex1} and \eqref{jfex2} only the  factors $a(0)$ and $b(-1)$ appear.
\end{remark}

\section{KS-functional}

\subsection{From scalar to matrix spectral function}

\begin{theorem}\label{thdensks} Let $A\in\GSMP(\bC)$.
Assume that $\sigma_+(\bc_j)=0,  \forall \bc_j\in\bC$. 
This spectral measure satisfies \eqref{eq5} if and only if the block Jacobi matrix $V(A_+)$ belongs to the Killip-Simon class.
\end{theorem}

A proof follows from the  lemma given below. Let
$$
V(z)=\sum_{j=1}^g\frac{\lambda_j}{\bc_j-z}
$$
and let $d\sigma$ be a scalar measure with an essential support on $E=V^{-1}([-2,2])$ such that $\sigma(\bc_j)=0$. We define the matrix measure $d\Sigma$ by
$$
\int \frac{d\Sigma(y)}{y-z}:=\int\frac 1 {V(x)-z}W^*(x)d\sigma(x)W(x),
$$
where
$$
 W(x)=
\begin{bmatrix}
\frac{1}{\bc_1-x} &\hdots&\frac{1}{\bc_g-x}
\end{bmatrix}.
$$

In other words, $d\Sigma$ is the matrix measure of the multiplication by $V(x)$ in $L^2_{d\sigma}$
with respect to a suitable cyclic subspace.
 Note that one can normalize this measure by  a triangular (constant) matrix $T$ such that
$$
T^*\int d\Sigma(y) T=I,
$$
that is, to choose an appropriate orthonormal basis in the fixed cyclic subspace.
\begin{lemma}
Let $\Sigma'(y)$ be the density of the a.c. part of the measure $d\Sigma$ on $[-2,2]$ and
$\sigma'(x)$ be the density of $d\sigma$, respectively. Then
\begin{equation}\label{densks}
 \det\Sigma'(y)=\frac{\prod_{V(x)=y}\sigma'(x)}{\prod_{k=1}^g\lambda_k}.
\end{equation}
\end{lemma}

\begin{proof}
Let $\{x_1,\dots,x_g\}=V^{-1}(y)$, $y\in[-2,2]$. Then
\begin{eqnarray*}
\Sigma'(y)&=&\sum_{V(x)=y}W^*(x)\frac{\sigma'(x)}{V'(x)}W(x)\\
&=&
\begin{bmatrix}
\frac{1}{\bc_1-x_1}&\dots&\frac 1{\bc_g-x_1}\\
\vdots&\dots&\vdots\\
\frac{1}{\bc_1-x_g}&\dots&\frac 1{\bc_g-x_g}
\end{bmatrix}^*
\begin{bmatrix}
\frac{\sigma'(x_1)}{V'(x_1)}& & \\
&\ddots& \\
& & \frac{\sigma'(x_g)}{V'(x_g)}
\end{bmatrix}
\begin{bmatrix}
\frac{1}{\bc_1-x_1}&\dots&\frac 1{\bc_g-x_1}\\
\vdots&\dots&\vdots\\
\frac{1}{\bc_1-x_g}&\dots&\frac 1{\bc_g-x_g}
\end{bmatrix}.
\end{eqnarray*}
As it is well known
\begin{equation}\label{densks1}
\det \begin{bmatrix}
\frac{1}{\bc_1-x_1}&\dots&\frac 1{\bc_g-x_1}\\
\vdots&\dots&\vdots\\
\frac{1}{\bc_1-x_g}&\dots&\frac 1{\bc_g-x_g}
\end{bmatrix}
=(-1)^{\frac{g(g-1)}{2}}\frac{\prod_{k<j}(x_k-x_j)\prod_{k<j}(\bc_k-\bc_j)}{\prod_{j,k}(\bc_j-x_k)}.
\end{equation}
On the other hand,
$$
y-V(x)=y\frac{\prod(x-x_j)}{\prod(x-\bc_j)}.
$$
Therefore,
$$
-V'(x_k)=y\frac{\prod_{k\not=j}(x_k-x_j)}{\prod_j(x_k-\bc_j)} \quad\text{and}\quad
-\lambda_k=y\frac{\prod_j(c_k-x_j)}{\prod_{k\not=j}(\bc_k-\bc_j)}.
$$
That is,
$$
V'(x_k)=\lambda_k\frac{\prod_{k\not=j}(x_k-x_j)}{\prod_j(x_k-\bc_j)}\frac{\prod_{k\not=j}(\bc_k-\bc_j)}{\prod_j(\bc_k-x_j)}.
$$
Thus,
\begin{equation}\label{densks2}
\prod V'(x_k)=\frac{\prod_{k<j}(x_k-x_j)^2(\bc_k-\bc_j)^2}{\prod_{k,j}(\bc_k-x_j)^2}\prod_k\lambda_k.
\end{equation}
Combining \eqref{densks1} and \eqref{densks2}, we obtain \eqref{densks}.
\end{proof}

\begin{proof}[Proof of Theorem \ref{thdensks}] Clearly, the eigenvalue spectral condition on $A_+$ corresponds to the eigenvalue spectral condition for $V(A_+)$ of the Killip-Simon class matrices with asymptotically constant matrix-block coefficients. By \eqref{densks}, we get the corresponding condition on the a.c. spectrum of $V(A_+)$.

\end{proof}

\subsection{``Derivative" in the Jacobi flow direction}

 Let us make the block decomposition of $V(A)$ in $(g+1)\times (g+1)$ blocks
 \begin{equation}\label{eq11}
 V(A)=\begin{bmatrix}\ddots&\ddots&\ddots & & &\\
 &\fv^*_{-1}&\fw_{-1}&\fv_0 & & \\
  & &\fv^*_{0}&\fw_{0}&\fv_1 & \\
& &  & \ddots&\ddots\ddots
 \end{bmatrix},
 \end{equation}
 where $\fw_k$ is a self-adjoint matrix and $\fv_k$ is  a lower triangular one, i.e.,
 $$
 \fw_k=\begin{bmatrix} w^{(k)}_{0,0}&\hdots &w^{(k)}_{0,g}\\ 
 \vdots&& \vdots \\
 w^{(k)}_{g,0}& \hdots&w^{(k)}_{g,g} \end{bmatrix},
\ \ 
 \fv_k=\begin{bmatrix} v^{(k)}_{0,0}&0 &0\\ 
 \vdots&\ddots& 0 \\
 v^{(k)}_{g,0}& \hdots&v^{(k)}_{g,g} \end{bmatrix}.
$$
Due to the previous subsection and general results on Jacobi block-matrices of Killip-Simon class \cite{KSDp},
 the spectral condition \eqref{eq5} is equivalent to the boundedness of
the following KS-functional
\begin{equation}\label{eq10}
H_+(A)=\frac 1 2\sum_{j\ge 0}\{ \tr (\fv_j^* \fv_j
+\fw_j^2
+\fv_{j+1} \fv_{j+1}^*)
-2(g+1)-\log\prod_{l=0}^g( v^{(j)}_{l,l}v^{(j+1)}_{l,l})^2\}.
\end{equation}

\begin{lemma}\label{lemjder}
Let
$$
\delta_J H_+(A)=\frac 1 2 \langle V(\cJ A) e_{-1}, V(\cJ A) e_{-1}\rangle-1-\log(\cJ v)_{g,g}^{(-1)}(\cJ v)_{g,g}^{(0)}.
$$
Then
$$
H_+(A)=H_+(\cJ A)+\delta_J H_+(A).
$$
\end{lemma}

\begin{proof}
We note that the form
$$
H_+(\cJ A)+\delta_J H_+(A)
$$
is related to the $P_+$ part of the matrix $U_A^*V(A)U_A$. Since $U(A)$  is of a block diagonal form, we can use the identities
\begin{alignat*}{3}
\tr\, U^*(\vp_j)\fv_j^*\fv_j U(\vp_j)&=~&\tr\,\fv_j^*\fv_j, \ \ \tr\, U^*(\vp_j)\fw^2_jU(\vp_j)=\tr\, \fw_j^2,\\
\tr\, U^*(\vp_j)\fv_{j+1}\fv_{j+1}^* U(\vp_j)&=&\tr\, \fv_{j+1}\fv_{j+1}.\qquad\qquad\qquad\qquad\qquad
\end{alignat*}
Similarly,
$$
\prod_{l=0}^g v^{(j)}_{l,l}=\det \fv_j=\det U^*(\vp_{j-1})\fv_j U(\vp_{j}).
$$

\end{proof}

\section{Proof of Theorem \ref{th73}}

\begin{lemma}\label{lem51}
Let $A\in\KSA(E,\bC)$ and $A(n+1)=\cJ A(n)$, $A(0)=A$. Then
\begin{equation}\label{cond1}
\{\lambda_0 p_g^{(-1)}{(n)}-1\}_{n\ge 0}\in\l_+^2,
\end{equation}
\begin{equation}\label{cond2}
\{ p_{m}^{(-1)}{(n)}-p_{m}^{(0)}{(n)}\}_{n\ge 0}\in\l_+^2,\
0\le m\le g-1,
\end{equation}
and
\begin{equation}\label{cond3}
\{ q_{m}^{(-1)}{(n)}-q_{m}^{(0)}{(n)}\}_{n\ge 0}\in\l_+^2,\
0\le m\le g.
\end{equation}
\end{lemma}

First we prove the following sublemma.
\begin{lemma}\label{lem52}
Assume that for sequences $\psi_n$ and $\tilde \psi_n$ there are sequences $\tau_n$ and $\tilde \tau_n$ such that
\begin{equation}\label{eqsubl}
\begin{bmatrix}
\tau_{n}&0\\
0
&1
\end{bmatrix}
\begin{bmatrix}
\sin\psi_{n}&\cos\psi_{n}\\
\cos\psi_{n}
&-\sin\psi_{n}
\end{bmatrix}-
\begin{bmatrix}
\sin\tilde\psi_{n}&\cos\tilde\psi_{n}\\
\cos\tilde\psi_{n}
&-\sin\tilde \psi_{n}
\end{bmatrix}\begin{bmatrix}
1&0\\0
&\tilde\tau_{n}
\end{bmatrix}\in \l^2_+,
\end{equation}
that is, all entries of the above matrix form $\l^2_+$-sequences. Assume in addition that there is $\eta>0$ such that
for all $n$
\begin{equation}\label{eqsubl1}
\cos\psi_n\ge\eta, \ \cos\tilde \psi_n\ge\eta, \ \frac 1\eta\ge \tau_n\ge\eta, \ \frac 1 \eta\ge\tilde\tau_n\ge\eta.
\end{equation}
Then $\{e^{i\psi_n}-e^{i\tilde\psi_n}\}_{n\ge 0}\in\l^2_+$.
\end{lemma}
\begin{proof}
Directly from \eqref{eqsubl} we have 
$$
\{\cos\psi_n-\cos\tilde\psi_n\}_{n\ge 0}\in\l^2_+\ \text{and}\  
\{\tau_n\cos\psi_n-\tilde \tau_n\cos\tilde\psi_n\}_{n\ge 0}\in\l^2_+.
$$
Then \eqref{eqsubl1} implies $\{\tau_n-\tilde\tau_n\}_{n\ge 0}\in\l^2_+$.
Now, we have another two conditions
$$
\{\tau_n\sin\psi_n-\sin\tilde\psi_n\}_{n\ge 0}\in\l^2_+\ \text{and}\  
\{\sin\psi_n-\tilde \tau_n\sin\tilde\psi_n\}_{n\ge 0}\in\l^2_+.
$$
Therefore,
$$
\sin{\psi_n}-\tau_n\tilde \tau_n\sin\psi_n-\tilde\tau_n(\sin\tilde\psi_n-\tau_n\sin\psi_n)
$$
belongs to $\l^2_+$, that is, $\{\sin{\psi_n}(1-\tau_n\tilde \tau_n)\}_{n\ge0}\in\l^2_+$. Thus, $(\tau_n^2-1)\sin\psi_n$ forms an $\l^2_+$-sequence, as well as $(\tau_n-1)\sin\psi_n$. Finally, since 
$$
\sin\psi_n-\sin\tilde\psi_n=\tau_n\sin\psi_n-\sin\tilde\psi_n-(\tau_n-1)\sin\psi_n,
$$
both $\{\sin\psi_n-\sin\tilde\psi_n\}_{n\ge 0}$ and $\{\cos\psi_n-\cos\tilde\psi_n\}_{n\ge 0}$ are $\l^2_+$-sequences.
\end{proof}

\begin{proof}[Proof of Lemma \ref{lem51}]
The relations \eqref{cond1} follows immediately from Lemma \ref{lemjder}. 

Let $\tilde A=\cO(A)$. We use tilde for all entries related to $\tilde A$ and $V(\tilde A)$, respectively. The entries of $A(n)$ we denote by $\{p_j^{(k)}(n),q_j^{(k)}(n) \}$ and we use a similar notation for the entries of $V(A(n))$ and $V(\tilde A(n))$. Due to Definition \ref{defo},
\begin{eqnarray}\label{eqar}\nonumber
\begin{bmatrix}
v^{(0)}_{g-1,g-1}(n)& 0\\
v^{(0)}_{g,g-1}(n)& \lambda_0 p^{(0)}_g(n)
\end{bmatrix}\begin{bmatrix}
\sin\phi_{g}^0(n)&\cos\phi_{g}^0(n)\\
\cos\phi_{g}^0(n)
&-\sin\phi_{g}^0(n)
\end{bmatrix}
\\
=
\begin{bmatrix}
\sin\phi_{g}^{-1}(n)&\cos\phi_{g}^{-1}(n)\\
\cos\phi_{g}^{-1}(n)
&-\sin\phi_{g}^{-1}(n)
\end{bmatrix}
\begin{bmatrix}
\lambda_0\tilde p^{(0)}_g(n)
& 0\\
\tilde w^{(0)}_{0,g}(n)& \tilde v^{(1)}_{0,0}(n)
\end{bmatrix}.
\end{eqnarray}
Applying  Lemma \ref{lemjder} to the matrix $A$, we obtain
$$
\{\lambda_0 p^{(0)}_g(n)-1\}_{n\ge 0}\in \l^2_+,\quad \{v^{(0)}_{g,g-1}(n)\}_{n\ge 0}\in \l^2_+.
$$
Similarly for the entries related to $\tilde A$ we have
$$
\{\lambda_0 \tilde p^{(0)}_g(n)-1\}_{n\ge 0}\in \l^2_+,\quad  \{\tilde w^{(0)}_{0,g}(n)\}_{n\ge 0}\in \l^2_+.
$$
 Thus, we can apply Lemma \ref{lem52} with respect to \eqref{eqar}. We get $\{\sin\phi_{-1}(n)-\sin\phi_0(n)\}$ belongs to $\l^2_+$. That is, $\{p^{(-1)}_{g-1}(n)-p^{(0)}_{g-1}(n)\}_{n\ge 0}\in\l^2_+$.
 
 Using \eqref{def3}, we get similar relations for all others $j$'s, $0\le j\le g-2$, i.e.: $\{p^{(-1)}_{j}(n)-p^{(0)}_{j}(n)\}_{n\ge 0}\in\l^2_+$.
 
 Using \eqref{def1} and \eqref{def4}, we prove \eqref{cond3}.
\end{proof}

%
\begin{proof}[Proof of Theorem \ref{th73}]
Lemma \ref{lem52} implies that 
$$
(v^{(-1)}_{g-1,g-1}(n)-1)\sin\phi^{-1}_g(n)
$$
form an $\l^2_+$-sequence, or, equivalently,
\begin{equation}\label{lambdal2}
\{(\Lambda^\#_{-1,g}(n)-\lambda_g) p^{(-1)}_{g-1}(n)\}_{n\ge 0}\in \l^2_+.
\end{equation}
Since $p^{(-1)}_{g-1}(n)$ may approach to zero, it does not imply yet that $\{\Lambda^\#_{-1,g}(n)-\lambda_g\}$ belongs to $\l ^2_+$. 
Let us show that
\begin{equation}\label{lambdal21}
\{(\Lambda^\#_{-1,g}(n)-\lambda_g) q^{(-1)}_{g-1}(n)\}_{n\ge 0}\in \l^2_+.
\end{equation}
Since $\inf_{n}\left((q^{(-1)}_{g-1}(n))^2+(p^{(-1)}_{g-1}(n))^2\right)>0$, both \eqref{lambdal2} and \eqref{lambdal21} give us
\eqref{m31} for $m=g$.

To this end, we note that 
\begin{equation}\label{lambdal22}
\Lambda^\#_{-1,g}(n+1)=\frac{\cos\phi^{-1}_g(n)}{\cos\phi^{-2}_g(n)}\Lambda^\#_{-1,g}(n).
\end{equation}
 Indeed, by definition
of the Jacobi flow
$$
U_{-2}\begin{bmatrix} v^{(-2)}_{g,g}(n+1)& & & \\
\ddots&v^{(-1)}_{0,0}(n+1) & &\\
 \ddots&\ddots & \ddots &\\
 \ddots&\ddots &\ddots &v^{(-1)}_{g-1,g-1}(n+1)
\end{bmatrix}=\fv_{-1}(n) U_{-1}
$$
the second from below entry in the last column in this matrix identity means exactly \eqref{lambdal22}.

Therefore, by Lemma \ref{lem52}, we get 
\begin{equation}\label{lambdal23}
\{\Lambda^\#_{-1,g}(n+1)-\Lambda^\#_{-1,g}(n)\}_{n\ge 0}\in \l^2_+. 
\end{equation}
Now,
by \eqref{jfex}
\begin{align*}
(\Lambda^\#_{-1,g}(n)-\lambda_g) p^{(-1)}_{g-1}(n)&=
-(\Lambda^\#_{-1,g}(n)-\lambda_g) q^{(-1)}_{g-1}(n+1)\\
&\times\frac{p_g^{(-1)}(n)\sqrt{\left(p_{g-1}^{(-1)}(n)\right)^2+\left(p_g^{(-1)}(n)\right)^2}}{
\|\vp_{-1}(n)\|}.
\end{align*}
In combination with \eqref{lambdal23} we have  \eqref{lambdal21}, and therefore \eqref{m31} for $m=g$.

The same arguments with respect to $\cO^m A$, $m=1,..,g-1$, give \eqref{m31} for all other $m$.

To show the opposite direction, we evaluate the entries of $V(A)e_{-1}$. Due to Lemma \ref{lem:gsmpEntries} and Lemma \ref{altda2}, we have
\begin{align*}
(h_{-1}(\bc_{k}))_{l}=&~\frac{(f_{-1})_l}{(f_{-1})_{k-1}}(h_{-1})_{k-1}\\
=\frac{1}{\bc_{l+1}-\bc_k}&\begin{bmatrix}
q^{(-1)}_{k-1}&-p^{(-1)}_{k-1}
\end{bmatrix}
\prod_{j=k}^{l-1}\fa(\bc_k,\bc_{j+1},p_j^{(-1)},q_j^{(-1)})
\begin{bmatrix}
p_l^{(-1)}\\
q_l^{(-1)}
\end{bmatrix}
\frac{\tilde\rho^{(0)}_{k-1}}{\Lambda^\#_{-1,k}},
\end{align*}
for $k<l+1$ and
\begin{align*}
 (h_{-1}(\bc_{l+1}))_{l}=-\frac{\tilde\rho^{(0)}_{k-1}}{\Lambda^\#_{-1,k}}.\qquad\qquad\qquad\qquad\qquad\qquad\qquad\qquad\qquad\qquad~
\end{align*}
Thus, 
\begin{align*}
w_{l,g}^{(-1)}&=\lambda_0p_g^{(-1)}q_l^{(-1)}+\sum_{k=1}^{l+1}\lambda_k(h_{-1}(\bc_{k}))_{l}\\
&=\lambda_0p_g^{(-1)}q_l^{(-1)}+\sum_{k=1}^{l}\Bigg(\frac{\lambda_k}{\Lambda^\#_{-1,k}}
\begin{bmatrix}
0 & 1
\end{bmatrix}
\prod_{j=0}^{k-2}\fa(\bc_k,\bc_{j+1};p^{(0)}_j,q_j^{(0)})\begin{bmatrix}p_{k-1}^{(0)}\\
q_{k-1}^{(0)}
\end{bmatrix}\\
&\times 
\prod_{j=k}^{l-1}\fa(\bc_k,\bc_{j+1},p_j^{(-1)},q_j^{(-1)})
\begin{bmatrix}
p_l^{(-1)}\\
q_l^{(-1)}
\end{bmatrix}
\frac{1}{\bc_{l+1}-\bc_k}\begin{bmatrix}
p^{(-1)}_{k-1}&q^{(-1)}_{k-1}
\end{bmatrix}\fj\Bigg)\\
&-\begin{bmatrix}
0 & 1
\end{bmatrix}
\prod_{j=0}^{l-1}\fa(\bc_{l+1},\bc_{j+1};p^{(0)}_j,q_j^{(0)})\begin{bmatrix}p_{l}^{(0)}\\
q_{l}^{(0)}
\end{bmatrix}.
\end{align*}
Using \eqref{m29}, \eqref{m30} and \eqref{m31}, we get $w_{l,g}^{(-1)}(n)\in\l^2_+$ by evaluating the identity 
\begin{align*}
\prod_{j=0}^{l-1}\fa(z,\bc_{j+1};p_j,q_j)=\sum_{k=1}^{l}\frac{\Res_{\bc_k} \prod_{j=0}^{l-1}\fa(z,\bc_{j+1};p_j,q_j)}{z-\bc_k}+I
\end{align*}
at the point $\bc_{l+1}$. $w^{(-1)}_{-1,g}(n)\in\l^2_+$ follows by \eqref{altd2bis} and the same considerations as in the proof of \eqref{explpqg} and Lemma \ref{lem:c0formula}. Similarly, one can show that $v_{g,g}^{(-1)}(n)$ and $v_{g,l}^{(0)}(n)$, for $0\leq l\leq g$, form $\l^2_+$ sequences.
\end{proof}

\section{Proof of the main Theorem \ref{mainhy}}
 
 \subsection{From GSMP to Jacobi}

Assume that $A\in\GSMP(\bC)$. Let $A(n)=\cJ^{\circ n} A$. Recall that the coefficients of the Jacobi matrix $J=\cF A$ are
given by \eqref{coefflow} and $\l^2$ properties of the coefficients $\{\vp_{-1}(n),\vq_{-1}(n),\vp_0(n),\vq_0(n)\}$ are given in Theorem \ref{th73}. We consider the isospectral surface $\is$ given by
$$
p_g=\frac 1{\lambda_0}, \quad q_g=-\bc_0-\lambda_0\sum_{j=1}^{g-1} p_jq_j, \quad \Lambda_k(\vp,\vq)=\lambda_k, \ k=1,...,g,
$$
with the identification $(p_j,q_j)\equiv(-p_j,-q_j)$, $j=0,\dots,g-1$.
Note that this is a $g$ dimensional torus, which we can parametrize by $\alpha\in\bbR^g/\bbZ^g	$ according to 
Theorem \ref{thm:multbyzsmp}. Moreover, for the given manifold
\begin{equation}\label{grad}
0<\inf_{\{\vp,\vq\}\in\is} \|T(\vp,\vq)\|\le \sup_{\{\vp,\vq\}\in\is} \|T(\vp,\vq)\|<\infty,
\end{equation}
where 
\begin{equation}\label{gradw}
T(\vp,\vq)=\begin{bmatrix}
\frac{\partial\Lambda_1}{\partial p_0}&\hdots &\frac{\partial\Lambda_1}{\partial q_{g-1}}\\
\vdots&\hdots&\vdots\\
\frac{\partial\Lambda_g}{\partial p_0}& &\frac{\partial\Lambda_g}{\partial q_{g-1}}
\end{bmatrix}.
\end{equation}

We define a periodic GSMP matrix $A(\alpha_n)$ generated by $\{\oc\vp(\alpha_n),\oc\vq(\alpha_n)\}\in \is$ such that
\begin{equation}\label{estmain1}
\dist(\{\vp_0(n),\vq_0(n)\},\is)=\dist(\{\vp_0(n),\vq_0(n)\},\{\oc\vp(\alpha_n),\oc\vq(\alpha_n)\}).
\end{equation}
By \eqref{m30}, \eqref{m31} and \eqref{grad}, we have
\begin{equation}\label{estmain}
\sum_{n=0}^\infty\dist^2(\{\vp_0(n),\vq_0(n)\},\{\oc\vp(\alpha_n),\oc\vq(\alpha_n)\})<\infty
\end{equation}
and also
\begin{equation}\label{estmain2}
a(n)^2-\cA(\alpha_n)\in\l^2,\quad b(n)-\cB(\alpha_n)\in\l^2.
\end{equation}

On the other hand, by \eqref{m29}-\eqref{m31} and the uniform smoothness of the Jacobi flow transform
\begin{eqnarray*}
\dist(\{\vp_0(n+1),\vq_0(n+1)\},\{\oc\vp(\alpha_n-\mu),\oc\vq(\alpha_n-\mu)\})
\\
\le C(E,J)
\dist(\{\vp_0(n),\vq_0(n)\},\{\oc\vp(\alpha_n),\oc\vq(\alpha_n)\})
\end{eqnarray*}
That is,
\begin{eqnarray*}
& &\dist(\{\oc\vp(\alpha_{n+1}),\oc \vq(\alpha_{n+1})\},\{\oc\vp(\alpha_n-\mu),\oc\vq(\alpha_n-\mu)\})
\\
&\le& C(E,J)
\dist(\{\vp_0(n),\vq_0(n)\},\{\oc\vp(\alpha_{n}),\oc\vq(\alpha_{n})\})
\\
&+&\dist(\{\vp_0(n+1),\vq_0(n+1)\},\{\oc\vp(\alpha_{n+1}),\oc\vq(\alpha_{n+1})\}).
\end{eqnarray*}
Since
$$
\|\alpha-\beta\|\le C_1(E)
\dist(\{\oc\vp(\alpha),\oc \vq(\alpha)\},\{\oc\vp(\beta),\oc\vq(\beta)\}),
$$
\eqref{estmain} implies
$$
\sum_{n=0}^\infty\|\epsilon_\alpha(n)\|^2<\infty, \quad \text{where}\quad \epsilon_\alpha(n):=\alpha_{n+1}-(\alpha_n-\mu).
$$
In combination with \eqref{estmain2}, we obtain \eqref{132} and \eqref{133}.

\begin{remark}\label{rem71}
Of course in this proof it is not necessary to choose $\alpha_n$ as the best approximation to $\{\vp(n),\vq(n)\}$,
see \eqref{estmain1}. It is enough to have the distance under an appropriate control. This explains a certain ambiguity in the representation \eqref{132}-\eqref{133}. 
\end{remark}

\subsection{From Jacobi to GSMP}\label{subs72}

In this section our goal is to estimate
$p_j(n)-\oc p_j(\alpha_n)$ and $q_j(n)-\oc q_j(\alpha_n)$
by means of the related distances $\dist((S^{-n}JS^n)_+, J(E))<\infty$.
In fact, we prove the following lemma. Note that \eqref{distalpn} evidently implies a word-by-word counterpart of \eqref{opp} in DKST, see Remark \ref{rem7}.

\begin{lemma}\label{lem72}
Let $J$ be  of the form \eqref{132}-\eqref{133}, then
\begin{equation}\label{distalpn}
\sum_{n=0}^\infty \dist^2_\eta((S^{-n}JS^n)_+,J(\alpha_n)_+)<\infty, \quad 
\alpha_n=\sum_{k=0}^n\epsilon_{\alpha}(k)-\mu n.
\end{equation}
\end{lemma}

\begin{proof}[Proof of Lemma \ref{lem72}]
We have
$$
|b(k+n)-\cB(\alpha_n-\mu k)|\le |\e_b(k+n)|+C_1(E)
\|\sum_{j=n+1}^{n+k}\epsilon_\alpha(j)
\|,
$$
where 
$$
C_1(E)=\sup_{\alpha\in\bbR^g/\bbZ^g}\|\text{\rm grad}\, \cB(\alpha)\|.
$$
For $\eta<1$, we have
$$
\sum_{n\ge 0} \left(\sum_{k\ge 1}
\|\sum_{j=n+1}^{n+k}\epsilon_\alpha(j)
\|^2\eta^{2k}\right)\le
\sum_{n\ge 0} \sum_{k\ge 1}\sum_{j=n+1}^{n+k}\|\epsilon_\alpha(j)\|^2k\eta^{2k}
$$
$$
=\sum_{k\ge 1} \sum_{n\ge 0}\sum_{j=n+1}^{n+k}\|\epsilon_\alpha(j)\|^2k\eta^{2k}\le
\sum_{k\ge 1} k\eta^{2k}\sum^\infty_{j\ge 1}k \|\epsilon_\alpha(j)\|^2
$$
$$
\le 
\sum^\infty_{j\ge 1}\|\epsilon_\alpha(j)\|^2
\sum_{k\ge 1} k^2\eta^{2k}.
$$
Making a similar estimation for $|a(k+n)^2-\cA(\alpha_n-\mu k)|$ we obtain \eqref{distalpn}. 
\end{proof}

Now, let $J=\cF A$, that is, $J=F^*AF$, where $F:\l^2\to\l^2$ is the unitary map such that  
$Fe_{-1}=e_{-1}$ and  $FP_+=P_+ F$, in particular, $Fe_0=\tilde e_0=\frac{1}{a(0)}P_+Ae_{-1}$.
We note that 
\begin{equation}\label{aort}
\{h=(A-\bc_1) f:\ f\in \l^2_+, \ \langle f,\tilde e_0 \rangle=0\}=\{h\in\l^2_+, \langle h, e_0 \rangle=0\}.
\end{equation}
Thus, $F^* e_0$ can be described by means of an orthogonal complement in the following construction.

Let $\bc\not\in \sigma(J)$ and $\zeta_\bc\in\bbD$ such that $\fz(\zeta_\bc)=\bc$. We assume that $\bc$ is real. We define
\begin{equation}\label{jort}
\l^2_{+,\bc}:=\{h=(J-\bc) f:\ f\in \l^2_+, \ \langle f,e_0 \rangle=0\}.
\end{equation}
Recall that
$r_+(z)=\langle (J_+-z)^{-1}e_0,e_0 \rangle$. 
\begin{lemma}
Let $\fK_\bc=\l^2_+\ominus \l^2_{+,\bc}$. This is a one dimensional space, i.e.,
$\fK_\bc=\{\k_{\bc}\}$. Moreover, we can choose
\begin{equation}\label{defkap}
\k_{\bc}=(J-\bc)^{-1}(e_{-1}a(0)\sin\varphi+e_0\cos\varphi),
\end{equation}
where
\begin{equation}\label{defkap2}
\tan\varphi=\tan\varphi(\bc)=r_+(\bc),\ -\frac{\pi}{2}<\varphi\le \frac{\pi}{2},
\end{equation}
including $\varphi=\frac{\pi} 2$ if $r_+(\bc)=\infty$, that is,
 $\bc$ is a pole of this function. 
In this notations
 \begin{equation}\label{defkap3}
\|\k_\bc\|^2=\frac{r'_+(\bc)}{1+r_+(\bc)^2}=\varphi'(\bc).
 \end{equation}
 Moreover, the following two-sided estimation holds
  \begin{equation}\label{defkap4}
 \frac{\min\{a(0)^2,1\}}{(|\bc|+\|J\|)^2}\le\varphi'(\bc)\le  \frac{\max\{a(0)^2,1\}}{\dist^2(\bc,\sigma(J))}.
\end{equation}
 \end{lemma}
 \begin{proof} If $r_+(\bc)\not=\infty$, we have $\k_\bc=(J_+-c)^{-1}e_0\cos\varphi$.
 Otherwise $\k_{\bc}$ is collinear to the corresponding eigenvector of $J_+$. These prove
 \eqref{defkap}, \eqref{defkap2}.
 
 Further, we have
 $$
 \|\k_{\bc}\|^2=\langle (J_+-\bc)^{-2}e_0,e_0 \rangle\cos^2\varphi=\frac{r'_+(\bc)}{1+r_+(\bc)^2},
 $$
 which proves \eqref{defkap3}.
 Now, let
$$
R(z):=\cE^*(J-z)^{-1}\cE=\int\frac{d\sigma}{x-z}.
$$
Since $\int d\sigma=I$, we have
$$
\frac{1}{(|\bc|+\|J\|)^2}
\le R'(\bc)=\int\frac{d\sigma}{(x-\bc)^2}\le \frac{1}{\dist^2(\bc,\sigma(J))}.
$$
Recall that
$$
R(z)=\begin{bmatrix}
r_{-}(z)^{-1}& a(0)\\
a(0)& r_+(z)^{-1}
\end{bmatrix}.
$$
That is,
$$
\frac{1}{(|\bc|+\|J\|)^2}
\le R(\bc)\begin{bmatrix}
\frac{r_-'(\bc)}{r_-(\bc)^2}&0
\\0&\frac{r_+'(\bc)}{r_+(\bc)^2}
\end{bmatrix}R(\bc)\le \frac{1}{\dist^2(\bc,\sigma(J))},
$$
or
$$
\frac{R(\bc)^{-2}}{(|\bc|+\|J\|)^2}
\le
\begin{bmatrix}
\frac{r_-'(\bc)}{r_-(\bc)^2}&0
\\0&\frac{r_+'(\bc)}{r_+(\bc)^2}
\end{bmatrix}\le \frac{R(\bc)^{-2}}{\dist^2(\bc,\sigma(J))}.
$$
Comparing the second entry of the second row of these matrices, we get
$$
\frac{1+a_0^2r_+(\bc)^2}{(|\bc|+\|J\|)^2}
\le
r_+'(\bc)
\le \frac{1+a_0^2r_+(\bc)^2}{\dist^2(\bc,\sigma(J))}.
$$
Thus, \eqref{defkap4} is also proved.
 \end{proof}
 
  Defining $\k_{\bc}$ by \eqref{aort} and \eqref{jort}, we obtain $F^*e_0=\frac{1}{\|\k_{\bc_1}\|}\k_{\bc_1}$.
Therefore,
 $$
 p^{(0)}_0(0)=\langle Ae_{-1},e_0 \rangle=\langle Je_{-1},\frac{\k_{\bc_1}}{\|\k_{\bc_1}\|} \rangle=\frac{a(0)\sin\varphi(\bc_1)}{\varphi'(\bc_1)}.
 $$
 The main estimations are based on the following lemma.
 \begin{lemma}
 Let $\oc J\in J(E)$. In the previous notations, 
 \begin{equation}\label{mlkap1}
\langle (J-\oc J)\k_{\bc},\oc\k_{\bc} \rangle=\sin(\oc\varphi(\bc)-\varphi(\bc)).
\end{equation}
Consequently, there exists $C=C(\sigma(J),\bc)<\infty$ such that
\begin{equation}\label{mlkap2}
|\sin(\oc\varphi(\bc)-\varphi(\bc))|\le C\dist_\eta(J_+,\oc J_+)
\end{equation}
and simultaneously for the  derivatives
\begin{equation}\label{mlkap3}
|(\oc\varphi)^{(m)}(\bc)-\varphi^{(m)}(\bc)|\le C\dist_\eta(J_+,\oc J_+)
\end{equation}
for $m=1,2,3$ and $\eta>|b(\bc)|$.
 \end{lemma}
 \begin{proof}
 We have
 $$
 \langle (J-\oc J)\k_{\bc},\oc\k_{\bc} \rangle=
 \langle (J-\bc)\k_{\bc},\oc\k_{\bc} \rangle-
 \langle (\oc J-\bc)\k_{\bc},\oc\k_{\bc} \rangle.
 $$
 We simplify the first term
 $$
  \langle e_{-1}a(0)\sin\varphi+e_0 \cos\varphi,\oc\k_{\bc} \rangle=
  \frac{1}{\oc a(0)}\langle (\oc J-c)e_{-1},\oc\k_{\bc} \rangle  \cos\varphi=\sin\oc\varphi\cos\varphi.
 $$
 Thus,
 $$
 \langle (J-\oc J)\k_{\bc},\oc\k_{\bc} \rangle=\sin\oc\varphi\cos\varphi
 -\sin\varphi\cos\oc \varphi
 $$
 and \eqref{mlkap1} is proved. 
 
  The upper estimation in \eqref{defkap4}
in combination with \eqref{mlkap1},
 \eqref{defkap3} implies
 \begin{equation}\label{mlkap7}
|\sin(\oc\varphi(\bc)-\varphi(\bc))|\le \sqrt{\varphi'(\bc)(\oc \varphi)'(\bc)}\frac{\| (J-\oc J)\oc\k_{\bc} \|}{\|\oc \k_{\bc}\|}
\le C\frac{\| (J-\oc J)\oc\k_{\bc} \|}{\|\oc \k_{\bc}\|}.
\end{equation}
Now, the vector $\frac 1{\|\oc\k_\bc\|}\oc\k_\bc$ in the functional model for $\oc J=J(\alpha)$ corresponds to the normalized reproducing kernel $\frac{1}{\| k^\alpha_{\zeta_\bc}\|} k^\alpha_{\zeta_\bc}$. The components of this vector were estimated in \eqref{estF}. Thus,
$$
\frac 1{\|\oc \k_{\bc}\|}{\| (J-\oc J)\oc\k_{\bc} \|}\le C(E)\dist_\eta(J_+,\oc J_+), \quad |b(\bc)|<\eta<1,
$$
and \eqref{mlkap7} implies \eqref{mlkap2}.

To get \eqref{mlkap3} we differentiate \eqref{mlkap1} with respect to $\bc$
\begin{equation}\label{firstde}
\cos(\oc\varphi(\bc)-\varphi(\bc))((\oc\varphi)'(\bc)-\varphi'(\bc))=\langle (J-\oc J)\k'_\bc,\oc\k_\bc \rangle+
\langle (J-\oc J)\k_\bc,(\oc\k_\bc)' \rangle.
\end{equation}
Since $\sin(\oc\varphi(\bc)-\varphi(\bc))$ was estimated from above, we have a uniform estimation for 
$|\cos(\oc\varphi(\bc)-\varphi(\bc))|$ from below. Using \eqref{defkap}, we evaluate $\k'_\bc$. Based on its explicit form and the estimation for $\varphi'_\bc$, we obtain that $\|k'_{\zeta_\bc}\|$ is also bounded by the distance from $\bc$ to $\sigma(J)$. Evidently, the coefficients of $(k^\alpha_{\zeta_\bc})'$ also satisfies \eqref{estF}. Thus,
$$
|(\oc\varphi)'(\bc)-\varphi'(\bc)|=\frac{\|(J-\oc J)\oc\k_c \|\|\k'_\bc\|+
\|(J-\oc J)(\oc\k_c)' \|\|\k_\bc\|}{|\cos(\oc\varphi(\bc)-\varphi(c))|}
$$
implies \eqref{mlkap3}. Taking the second and third derivatives in \eqref{firstde}, we obtain \eqref{mlkap3} for $m=2,3$.
 \end{proof}

\begin{corollary} If $J$ is of the form \eqref{132}-\eqref{133} and $A(n)=\cF^{-1}(S^{-n} J S^n)$,
 then
 \begin{equation}\label{zzz}
\sum_{n=0}^\infty|p^{(0)}_0(n)-p_0(\alpha_n)|^2<\infty, \quad \alpha_n= \sum_{k=0}^n\phi_k-\mu n.
\end{equation}

\end{corollary}
\begin{proof}
Thus, by \eqref{mlkap2}, \eqref{mlkap3} we can estimate the difference
$$
p^{(0)}_0(n)-p_0(\alpha_n)=\frac{a(n)\sin\varphi(\bc_1)}{\varphi'(\bc_1)}-
\frac{\oc a(0)\sin\oc\varphi(\bc_1)}{(\oc\varphi)'(\bc_1)},\quad \oc J=J(\alpha_n),
$$
by means of $\dist((S^{-n}J S^n)_+,J(\alpha_n)_+)$. Due to \eqref{distalpn}, we have \eqref{zzz}.
\end{proof}

\begin{proof}[Finishing the proof of Theorem \ref{mainhy}]
It remains to show that \eqref{132}-\eqref{133} imply \eqref{m29}-\eqref{m31}.

Consider the ordered system of vectors
\begin{equation}\label{mthos}
e_{-1},\k_{\bc_1},\dots,\k_{\bc_g},e_0,\k'_{\bc_1},\dots,\k'_{\bc_g},e_1.
\end{equation}
Let us point out that the orthogonalization of the system
\begin{equation}\label{mthos1}
e_{-1},\oc\k_{\bc_1},\dots,\oc\k_{\bc_g},e_0,(\oc\k_{\bc_1})',\dots,(\oc \k_{\bc_g})',e_1
\end{equation}
leads to the family $\{f^\alpha_j\}_{j=-1}^{2g+2}$, see \eqref{smpbase}, where $\oc J=J(\alpha)$.

To evaluate the Gram-Schmidt matrix of the system \eqref{mthos} we use
$$
\langle \k_{\bc_j},\k_{\bc_m} \rangle=\frac{r_+(\bc_j)-r_+(\bc_m)}{\bc_j-\bc_m}\cos\varphi(\bc_j)\cos\varphi(\bc_m)=
\frac{\sin(\varphi(\bc_j)-\varphi(\bc_m))}{\bc_j-\bc_m}.
$$
Therefore,
$$
\langle \k'_{\bc_j},\k_{\bc_m} \rangle=\frac{\cos(\varphi(\bc_j)-\varphi(\bc_m))}{\bc_j-\bc_m}\varphi'(\bc_j)-
\frac{\sin(\varphi(\bc_j)-\varphi(\bc_m))}{(\bc_j-\bc_m)^2}
$$
and
\begin{eqnarray*}
\langle \k'_{\bc_j},\k'_{\bc_m} \rangle&=&
\frac{\cos(\varphi(\bc_j)-\varphi(\bc_m))}{(\bc_j-\bc_m)^2}(\varphi'(\bc_j)+\varphi'(\bc_m))
\\
&+&
\frac{\sin(\varphi(\bc_j)-\varphi(\bc_m))}{\bc_j-\bc_m}\varphi'(\bc_j)\varphi'(\bc_m)
-2\frac{\sin(\varphi(\bc_j)-\varphi(\bc_m))}{(\bc_j-\bc_m)^3}
\end{eqnarray*}
for $j\not=m$.

Having uniform estimations from below for all Gram-Schmidt determinants of the system \eqref{mthos1},
from  \eqref{mlkap2}, \eqref{mlkap3}, similarly to \eqref{zzz}, we obtain
$$
\sum_{n=0}^\infty|p^{(m)}_j(n)-p_j(\alpha_n)|^2<\infty,
\ \sum_{n=0}^\infty|q^{(m)}_j(n)-q_j(\alpha_n)|^2<\infty,\quad m=0,1,
$$
for all $j=0,\dots,g.$ This implies \eqref{m29}-\eqref{m31}, in particular,
$$
\sum_{n=0}^\infty|p^{(1)}_j(n)-p^{(0)}_j(n)|^2<\infty,\quad\text{for}\ j=0,..,g-1.
$$

\end{proof}


 \bibliographystyle{amsplain}
 \bibliography{lit2}

\providecommand{\bysame}{\leavevmode\hbox to3em{\hrulefill}\thinspace}
\providecommand{\MR}{\relax\ifhmode\unskip\space\fi MR }
\providecommand{\MRhref}[2]{%
  \href{http://www.ams.org/mathscinet-getitem?mr=#1}{#2}
}
\providecommand{\href}[2]{#2}
\begin{thebibliography}{10}

\bibitem{ALF}
L.~V. Ahlfors, \emph{{Bounded analytic functions}}, Duke. Math. J. \textbf{14}
  (1947), 1--11.

\bibitem{AKH}
N.~I. Akhiezer, \emph{{A Generalization of a Minimal Problem of
  Korkin-Zolotarev kind}}, Academic Press \textbf{4} (1936), no.~XIII.

\bibitem{Akh60}
\bysame, \emph{{Orthogonal polynomials on several intervals}}, Soviet Math.
  Dokl. (1936), 989--992.

\bibitem{AKHmp}
\bysame, \emph{{The classical moment problem and some related questions in
  analysis}}, Hafner Publishing Co., New York, 1965.

\bibitem{AKHef}
\bysame, \emph{{Elements of the Theory of Elliptic Functions}}, Amer. Math.
  Soc., Providence, 1990.

\bibitem{BD2}
Y.~M. Berezansky and M.~E. Dudkin, \emph{The strong {H}amburger moment problem
  and related direct and inverse spectral problems for block {J}acobi-{L}aurent
  matrices}, Methods Funct. Anal. Topology \textbf{16} (2010), no.~3, 203--241.

\bibitem{jreview}
J.~S. Christiansen, B.~Simon, and M.~Zinchenko, \emph{{Finite gap Jacobi
  matrices: A review.}}, Proc. Sympos. in Pure Math. \textbf{87} (2013),
  87--103.

\bibitem{CD}
M.~J. Cowen and R.~G. Douglas, \emph{{Complex geometry and operator theory}},
  Acta Math \textbf{141} (1978), 187--261.

\bibitem{KSDp}
D.~Damanik, R.~Killip, and B.~and Simon, \emph{{Perturbations of orthogonal
  polynomials with periodic recursion coefficients}}, Annals of Math.
  \textbf{171} (2010), no.~3.

\bibitem{DK}
P.A. Deift and R.~Killip, \emph{{On the absolutely continuous spectrum of
  one-dimensional Schr\"odinger operators with square summable potentials}},
  Comm. Math. Phys. \textbf{203} (1999), 341--347.

\bibitem{DUD}
M.~E. Dudkin, \emph{The inner structure of the {J}acobi-{L}aurent matrix
  related to the strong {H}amburger moment problem}, Methods Funct. Anal.
  Topology \textbf{19} (2013), no.~2, 97--107.

\bibitem{EPY}
B.~Eichinger, F.~Puchhammer, and P.~Yuditskii, \emph{{Jacobi Flow on SMP
  Matrices and Killip-Simon Problem on Two Disjoint Intervals}}, submitted to
  Computational Methods and Function Theory.

\bibitem{Fay}
J.~Fay, \emph{{Theta Functions on Riemann Surfaces}}, Lecture Notes in
  Mathematics, Springer-Verlag, 1970.

\bibitem{GNR}
F.~Gamboa, J.~Nagel, and Rouault A., \emph{{Sum rules via large deviations}},
  arXiv: 1407.1384 (2014).

\bibitem{GZ}
L.~Golinski and A.~Zlatos, \emph{{Coefficients of orthogonal polynomials on the
  unit circle and higher-order Szeg\"o theorems}}, Constr. Approx. \textbf{26}
  (2007), no.~3.

\bibitem{Has}
M.~Hasumi, \emph{{Hardy Classes on Infinitely Connected Riemann Surfaces}},
  Lecture Notes in Math., Springer, 1983.

\bibitem{HN}
E.~Hendriksen and C.~Nijhuis, \emph{Laurent-{J}acobi matrices and the strong
  {H}amburger moment problem}, Proceedings of the {I}nternational {C}onference
  on {R}ational {A}pproximation, {ICRA}99 ({A}ntwerp), vol.~61, 2000,
  pp.~119--132.

\bibitem{JN}
W.~B. Jones and O.~Nj{\aa}stad, \emph{{Orthogonal Laurent polynomials and
  strong moment theory: a survey. Continued fractions and geometric function
  theory}}, J. Comput. Appl. Math. \textbf{105} (1999), no.~1-2.

\bibitem{KS}
R.~Killip and B.~Simon, \emph{{Sum rules for Jacobi matrices and their
  applications to spectral theory}}, Annals of Math. \textbf{158} (2003),
  no.~2.

\bibitem{K2004}
S.~Kupin, \emph{{On a spectral property of Jacobi matrices}}, Proc. Amer. Math.
  Soc. \textbf{132} (2004), no.~5.

\bibitem{LNS}
A.~Laptev, S.~Naboko, and O.~Safronov, \emph{{On new relations between spectral
  properties of Jacobi matrices and their coefficients}}, Comm. Math. Phys.
  \textbf{241} (2003), no.~1.

\bibitem{LU}
M.~Lukic, \emph{{On a conjecture for higher-order Szeg\"o theorems}}, Constr.
  Approx. \textbf{38} (2013), 161--169.

\bibitem{MaT}
V.~Matveev, \emph{{30 years of finite-gap integration theory}}, Phil. Trans. R.
  Soc. A \textbf{366} (2008).

\bibitem{MUM}
D.~Mumford, \emph{{Tata lectures on theta, vol. I, II.}}, MA: Birkh\"auser,
  Boston, 1983.

\bibitem{NPVY}
F.~Nazarov, F.~Peherstorfer, A.~Volberg, and P.~Yuditskii, \emph{{On
  generalized sum rules for Jacobi matrices}}, Int. Math. Res. Not. (2005),
  no.~3, 155--186.

\bibitem{PVY}
F.~Peherstorfer, A.~Volberg, and P.~Yuditskii, \emph{{CMV matrices with
  asymptotically constant coefficients. Szego-Blaschke class, scattering
  theory}}, Journal of Functional Analysis \textbf{256} (2009), 2157--2210.

\bibitem{PY}
F.~Peherstorfer and P.~Yuditskii, \emph{{Asymptotic behaviour of polynomials
  orthonormal on a homogeneous set}}, J. Anal. Math. \textbf{89} (2003),
  113--154.

\bibitem{Pom}
Ch. Pommerenke, \emph{{On the Green's function of Fuchsian groups}}, Ann. Acad.
  Sci. Fenn. \textbf{2} (1976), 409--427.

\bibitem{POT}
V.P. Potapov, \emph{{The Multiplicative Structure of J-contractive Matrix
  Functions}}, American Mathematical Society translations, American
  Mathematical Society, 1960.

\bibitem{REMA11}
C.~Remling, \emph{{The absolutely continuous spectrum of Jacobi matrices}},
  Annals of Math. \textbf{174} (2011), no.~2, 125--171.

\bibitem{2005v1}
B.~Simon, \emph{{Orthogonal polynomials on the unit circle. Part 1. Classical
  theory }}, American Mathematical Society Colloquium Publications, American
  Mathematical Society, Providence, 2005.

\bibitem{2005v2}
\bysame, \emph{{Orthogonal polynomials on the unit circle. Part 2. Spectral
  theory}}, American Mathematical Society Colloquium Publications, American
  Mathematical Society, Providence, 2005.

\bibitem{BS}
\bysame, \emph{{Szeg\"o's Theorem and Its Descendants: Spectral Theory for
  $L^2$ Perturbations of Orthogonal Polynomials}}, Princeton University Press,
  New Jersey, 2011.

\bibitem{SZ}
B.~Simon and A.~Zlatos, \emph{{Higher-order Szeg\"o theorems with two singular
  points}}, J. Approx. Theory \textbf{134} (2005), no.~1, 114--129.

\bibitem{Sim1}
K.~Simonov, \emph{{Orthogonal Matrix Laurent Polynomials}}, Mathematical Notes
  \textbf{79} (2006), no.~2, 291--295.

\bibitem{Sim2}
\bysame, \emph{{Strong matrix moment problem of Hamburger}, methods of
  functional analysis and topology}, Mathematical Notes \textbf{12} (2006),
  no.~2, 183--196.

\bibitem{SY}
M.~Sodin and P.~Yuditskii, \emph{{Almost periodic Jacobi matrices with
  homogeneous spectrum, infinite-dimensional Jacobi inversion, and Hardy spaces
  of character-automorphic functions}}, J. Geom. Anal. \textbf{7} (1997),
  387--435.

\bibitem{GT}
G.~Teschl, \emph{{Jacobi Operators and Completely Integrable Nonlinear
  Lattices}}, Mathematical surveys and monographs, vol.~72, American
  Mathematical Society, Providence.

\bibitem{VY}
A.~Volberg and P.~Yuditskii, \emph{{On the inverse scattering problem for
  Jacobi matrices with the spectrum on an interval, a finite system of
  intervals or a Cantor set of positive length}}, Commun. Math. Phys.
  \textbf{226} (2002), no.~3, 567--605.

\bibitem{vN}
J.~von Neumann, \emph{{Charakterisierung des Spektrums eines
  Integraloperators}}, Actualit\'es Sci. Indust. \textbf{229} (1935).

\bibitem{WID}
H.~Widom, \emph{{$H_p$ sections of vector bundles over Riemann surfaces}}, Ann.
  Math. \textbf{94} (1971), 304--324.

\end{thebibliography}

\bigskip

{Institute for Analysis, Johannes Kepler University Linz,
A-4040 Linz, Austria} 

\noindent\textit{E-mail address}:
{petro.yudytskiy@jku.at,}\\
\textit{E-mail address}:
{benjamin.eichinger@jku.at.}
\end{document}